\documentclass[11pt,a4paper,english]{amsart}
\usepackage{amssymb,amsmath,amsthm,babel,enumerate,ulem}
\usepackage[unicode=true,
 bookmarks=true,bookmarksnumbered=true,bookmarksopen=false,
 breaklinks=false,pdfborder={0 0 1},backref=false,colorlinks=false]{hyperref}
\newtheorem{theorem}{Theorem}[section]
\newtheorem*{theorem*}{Theorem}
\newtheorem{proposition}[theorem]{Proposition}

\newtheorem*{definition*}{Definition}
\newtheorem{assumption}[theorem]{Assumption}
\newtheorem{lemma}[theorem]{Lemma}
\newtheorem*{lemma*}{Lemma}

\newtheorem{corollary}[theorem]{Corollary} 
 
\newtheorem*{remark*}{Remark}
\newtheorem{convention}[theorem]{Convention}

\newcommand{\Pt}{P}%{t}%A point

\newcommand{\N}{{\mathbb N}}
\newcommand{\Z}{{\mathbb Z}}
\newcommand{\Q}{{\mathbb Q}}
\newcommand{\C}{{\mathbb C}}
\newcommand{\R}{{\mathbb R}}
\newcommand{\Cu}{{\mathcal C}}
\newcommand{\F}{{\mathbb F}}

\renewcommand{\L}{{\mathcal L}}
\newcommand{\Qb}{\overline{\mathbb Q}}

\newcommand{\A}{{\mathbb A}}
\renewcommand{\P}{{\mathbb P}}
\newcommand{\Gm}{{\mathbb G}_{\rm m}}
\newcommand{\Norm}{{\rm Norm}}

\newcommand{\rank}{{\rm rank\,}}
\renewcommand{\div}{{\rm div}}

\newcommand{\ord}{{\rm ord}}
\newcommand{\ie}{{\it i.\,e.\ }}
\newcommand{\e}{{\varepsilon}}
\newcommand{\cf}{{\it cf}~}
\renewcommand{\O}{{\mathcal O}}

\newcommand{\rhog}{{\boldsymbol\rho}}
\newcommand{\alphag}{{\boldsymbol\alpha}}

\newcommand{\omegag}{{\boldsymbol\omega}}
\newcommand{\gammag}{{\boldsymbol\gamma}}

\newcommand{\etag}{{\boldsymbol\eta}}
\renewcommand{\thetag}{{\boldsymbol\theta}}

\newcommand{\psig}{{\boldsymbol\psi}}

\newcommand{\ug}{{\bf u }}

\newcommand{\xg}{{\bf x }}
\newcommand{\yg}{{\bf y }}
\newcommand{\Xg}{{\bf X }}

\newcommand{\fb}{{\bf f }}
\newcommand{\gb}{{\bf g }}
\newcommand{\1}{{\bf 1 }}

\renewcommand{\mid}{\,\vert\,}

\newcommand{\stkout}[1]{\ifmmode\text{\sout{\ensuremath{#1}}}\else\sout{#1}\fi}

\usepackage{color}

\numberwithin{equation}{section}

\title{Pencils of norm form equations and a conjecture of Thomas, II.}
\author{F. Amoroso, D. Masser, U. Zannier}
\date{August 20, 2026}
\begin{document}

\maketitle
%\date{}

\vglue1cm
\noindent
{\bf Abstract.} We continue our studies on parametric norm forms $F_t({\bf x})$, with ${\bf x}=(x_0,x_1,\ldots,x_{d-1})$ lying in some parametric linear subvariety $W_t$ and integers $t$ sufficiently large. In \cite{Am-Ma-Za2} we proved some effective specialization results for integer solutions $\bf x$ of $F_t({\bf x})=1$. Here we modify our techniques to treat $F_t({\bf x})=q$ for an arbitrary integer $q$. Under mild conditions (not however including the crucial index assumption in \cite{Am-Ma-Za2}) we show that all $\bf x$ are polynomially bounded in terms of $|q|$ and $t$. As in \cite{Am-Ma-Za2} we use the methods of our \cite{Am-Ma-Za} based on diophantine approximation techniques to bound certain heights. In particular we do not use linear forms in logarithms and indeed it seems unlikely that those can lead to such polynomial bounds, even for Thue equations in two variables with $x_2=\cdots=x_{d-1}=0$. We present an example with eight variables.

\noindent
{\tt 2020 MSC codes: 11D57, 11G50.} 

\section{Introduction}
\label{intro}
Let $P\in \Z[T,X]$ be an irreducible polynomial, monic of degree $d$ in $X$, and let $\xi\in\overline{\Q(T)}$ be one of its roots. We fix a specialization map $\Q\to\Qb$ which for any $u\in\Q(\xi)$ sends $t\in\Q$ (we exclude the finitely many ramified ones) to $u_t$; see \cite{Am-Ma-Za2}, beginning of section~\ref{sec:ex}, for a precise definition (for large positive $t$ one can also use the Puiseux series - see for example just after \ref{L} below). The basic parametric equation in~\cite{Am-Ma-Za2} was 
\begin{equation}
\label{NormForm}
\Norm(x_0+x_1\xi_t+\cdots+x_{d-1}\xi_t^{d-1})=1
\end{equation}
to be solved in suitably restricted integers $x_0,\ldots,x_{d-1}$ as the integer parameter $t$ varies. We shall soon see that under certain circumstances $[\Q(\xi_t):\Q]=d$ for all sufficiently large positive integers $t$. At any rate it is not hard to see that for $t\in\N$ large the numbers $\varrho_1$ and $2\varrho_2$ of real and non-real zeroes of $P(t,X)$ are both constants. This can be seen by looking at sign changes, as in \cite[Lemma 3.1]{Am-Ma-Za2} or by considering Puiseux series at $T=\infty$ (for example $\varrho_1$ is the number of such series with real coefficients). We write 
$$
\varrho=\varrho_1+\varrho_2-1.
$$ 
and we make the following
\begin{assumption}~
\label{ass:Thomas-H}
\begin{enumerate}[(1)]
\item There exist multiplicatively independent elements $u_1,\ldots,u_\varrho$ of $\Z[T,\xi]^*$.
%\item {\red The rank of $\Z[T,\xi]^*$ and the rank of the group of units of $\Q(\xi_t)$ are the same for large $t$.}
\item Suppose that $u$ is in $\Z[T,\xi]^*$ and the algebraic number $v$ is in the group generated by the conjugates of $u$. Then $v$ is a root of unity.
\end{enumerate}
\end{assumption}

This is a slightly corrected version of Assumption 2.1 of \cite{Am-Ma-Za2}. In a short appendix we list the modifications needed in that paper.

Let $(\Z[T,\xi]^*)_t=\{u_t,\; u\in\Z[T,\xi]^*\}$. Note that $(\Z[T,\xi]^*)_t\subseteq \Z[\xi_t]^*$. Under our assumptions, the index 
$$
[\Z[\xi_t]^*:(\Z[T,\xi]^*)_t]
$$ 
is finite for large $t$ and moreover uniformly bounded (\cite[Theorem 2.3]{Am-Ma-Za2} see appendix).\\

As a consequence of the specialization theorem~\cite[Theorem 1.3]{Am-Ma-Za} we can state our earlier result as follows.

\begin{theorem}{\rm (Theorem 2.2 of \cite{Am-Ma-Za2}).}
\label{thm:AMZ2-Thm2.2}
Let us assume~\ref{ass:Thomas-H}. Let $W\subseteq\A^d$ be a proper subvariety defined over $\Q(T)$. Then there exists an effective $t_0>0$ such that for $t\geq t_0$ with 
\begin{equation}
\label{eq:index}
\Z[\xi_t]^*=(\Z[T,\xi]^*)_t,
\end{equation} 
the solutions $(x_0,\ldots,x_{d-1})\in W_t(\Z)$ of 
\begin{equation}
\label{norm-spe0}
\prod_{j=1}^d(x_0+x_1\sigma_j(\xi_t)+\cdots+x_{d-1}\sigma_j(\xi_t)^{d-1})=1
\end{equation}
are specializations of functional solutions $\Xg=(X_0,\ldots,X_{d-1})\in W(\Z[T])$ of 
\begin{equation*}
\prod_{j=1}^d(X_0+X_1\sigma_j(\xi)+\cdots+X_{d-1}\sigma_j(\xi)^{d-1})=1.
\end{equation*}
\end{theorem}

Next we proved (\cite[Theorem 2.4]{Am-Ma-Za2}) a uniform bound for the index valid for {\sl all} sufficiently large~$t$. This bound allowed us to prove a completely (\ie without assumption on the index) effective description of the solutions of the norm form equation~\eqref{NormForm} at the price of a more technical statement.\\

In this paper we modify the above Theorem~\ref{thm:AMZ2-Thm2.2} in several ways.

(i) On a significant technical level we are able to eliminate the condition~\eqref{eq:index}, which can be rather tedious (at the very least - see more below) to verify even for specific examples with $d=5$. 

(ii) Then we generalize the equation~\eqref{norm-spe0} to allow an arbitrary integer $q$ on the right-hand side. 

(iii) But crucially we estimate from above all solutions of the resulting equation, instead of relating them to specialization of functional solutions. This brings us much closer to the classical diophantine theory of norm form equations. 

(iv) Furthermore our estimates are effective, which cannot be achieved through a standard tool like the Subspace Theorem. 

(v) The estimates are even polynomial in $|q|$ and $t$, which we do not see how to achieve through the other standard tool of Linear Forms in Logarithms even for specific examples with $d=3$. 

(vi) Finally we assume that $W$ is linear, $d$ is odd, and we impose an even milder requirement about the extension ${\Q}(T,\xi)/{\Q}(T)$ which holds in some sense generically.

Regarding (ii) and (iii) above it seems unlikely that there is a specialization result for \eqref{norm-spe0} with 1 replaced by general $q$ in $\Z$. But possibly there is with 1 replaced by $Q(t)$ for any fixed non-zero $Q$ in $\Z[T]$.

Anyway, here is the main result of our paper.

\begin{theorem}
\label{thm:main}
Let us assume~\eqref{ass:Thomas-H}. Let also assume that $\Q(T,\xi)/\Q(T)$ is primitive (\ie that it does not contain non trivial subextensions) of odd degree.
%the Galois group $G$ of the normal closure of $\Q(T,\xi)/\Q(T)$ acts double transitively. 
Let $W\subsetneq\A^d$ be a linear subvariety defined over $\Q(T)$. Then there exist $c,C,t_0\geq1$ effectively depending only on $\xi$ and $W$ which satisfy the following assertion.\par 
Let $q$ be a non-zero integer and let $\xg=(x_0,\ldots,x_{d-1})\in W_t(\Z)$ be a solution of 
\begin{equation}
\label{eq:norm-spe}
\prod_{j=1}^d(x_0+x_1\sigma_j(\xi_t)+\cdots+x_{d-1}\sigma_j(\xi_t)^{d-1})=q
\end{equation}
for some positive integer $t$. Then if $t \geq t_0$ we have
\begin{equation}
\label{eq:bound-sol}
\max\{|x_0|,|x_1|,\ldots,|x_{d-1}|\} \leq t^{C}|q|^{c}.
\end{equation}
\end{theorem}

We should emphasize that even the finiteness of solutions of~\eqref{eq:norm-spe}, for fixed $t$ however large, is very far from obvious. The only previous proof that we know comes from the Subspace Theorem; see for example Theorem 10B of Schmidt~\cite{sch} (p.242). It is not difficult to check that his conditions for finiteness are implied by our conditions of primitivity and odd $d$. But of course his finiteness is ineffective.

We will see in section~\ref{sec:ex2} that the primitivity indeed has a generic character. 

And section~\ref{sec:ss} contains a discussion concerning the solution with the maximum in~\eqref{eq:bound-sol}.\\

We have reason to believe that the bound~\eqref{eq:bound-sol} is essentially best possible regarding the dependence on $|q|$ and $t$. Let us consider the equation
\begin{equation}
\label{eq:ge}
x^3-(t^3-1)y^3=q
\end{equation}
which was already (with $q=1$) our guiding example in [2]. If $q=p^3$ it has the solution $(x,y)=(tp,p)$ with
$$\max\{|x|,|y|\} = t|p|=t|q|^{1/3}.$$

This~\eqref{eq:ge} also allows us a pause to explain the novelty of our method. In~\cite{Za} (pp. 349-352) the third author treats~\eqref{eq:ge} using hypergeometric functions (as introduced by~\cite{Be} - these will be used also in our section~\ref{sec:ss} for a different purpose). Here we sketch how to treat~\eqref{eq:ge} with our new method.
%The easy basic case of~\eqref{eq:norm-spe} appears in~\cite[Exercise 5.13, p.349]{Za} where the author considers 
%the equation 
%\begin{equation}
%\label{eq:ge}
%x^3-(t^3-1)y^3=q
%\end{equation}
%which was already (with $q=1$) our guiding example in \cite{Am-Ma-Za}. In~\cite{Za} the author focuses on the closely related issue of irrationality measures for $\sqrt[3]{t^3-1}$. For the reader's convenience, let us briefly sketch how we can treat with this method the norm form equation~\eqref{eq:ge}.
Let $x,y,t$ satisfy~\eqref{eq:ge}. With $\xi=\root 3 \of {T^3-1}$ this says that $x-\xi_ty$ has norm $q$. Here $\gamma_t=t-\xi_t$ is a unit in $\Q(\xi_t)$. So by a fairly standard ``fundamental domain'' argument (see Lemma \ref{lem:box} for details) we can find an integer $n$ such that
\begin{equation}
\label{eq:beta}
\beta=\gamma_t^{-n}(x-\xi_ty)
\end{equation}
has height
\begin{equation}
\label{eq:hbeta}
h(\beta) \ll \lambda
\end{equation}
for $\lambda=\log(t|q|)$ and absolute implied constants (and from now on).

We note that $\mu=\log\max\{|x|,|y|\}$ satisfies
$$
\mu \ll h(x-\xi_ty)+h(x-\omega\xi_ty)+\log t
$$
with a primitive cube root $\omega$ of unity, and so by~\eqref{eq:beta},~\eqref{eq:hbeta} this gives
\begin{equation}
\label{eq:mu}
\mu \ll \lambda+(|n|+1)\log t.
\end{equation}
With the maps sending $\alpha$ to $\alpha',\alpha''$ of $\Q(\xi_t)$ defined by $\xi_t'=\omega\xi_t,\xi_t''=\omega^2\xi_t$ we have the classical ``Siegel Identity''
\begin{multline}
\label{eq:siegel}
\beta\gamma_t^n+\omega\beta'{\gamma'_t}^n+\omega^2\beta''{\gamma_t''}^n\\=(x-\xi_ty)+\omega(x-\omega\xi_ty)+\omega^2(x-\omega^2\xi_ty)=0.
\end{multline}
Traditional now is to use linear forms in logarithms (for more on the problems with this see section~\ref{sec:ex}).

If $n \geq n_0$ for some large absolute $n_0>0$ then we use instead one of the main results of~\cite{Am-Ma-Za2}, reproduced as Theorem~\ref{thm:AMZ-Thm1.4} below. This gives
$$\log t = h(t) \ll {\lambda \over n}+1.$$
Thus by~\eqref{eq:mu}
$$\mu \ll \lambda+n\log t \ll \lambda+n.$$
But~\eqref{eq:beta} gives again
$$n\log t \ll nh(\gamma_t)=h\left({x-\xi_ty \over \beta}\right)\ll \mu+\lambda.$$
Thus
$$\mu \ll \lambda+{\mu+\lambda \over \log t}.$$
So if $t \geq t_0$ we conclude
\begin{equation}
\label{eq:gea}
\log\max\{|x|,|y|\}=\mu \ll \lambda=\log(t|q|)
\end{equation}
as in Theorem~\ref{thm:main}.

And if $n \leq -n_0$ then we simply replace $n,\gamma_t$ in~\eqref{eq:siegel} by $-n,\gamma_t^{-1}$.

Finally if $|n| < n_0$ then~\eqref{eq:mu} gives the result at once.

We remark that we did not use~\eqref{eq:index}; indeed we never explicitly calculated $\Z[\xi_t]^*$ or even the generic $\Z[T,\xi]^*$. In fact the former has rank one and is generated by $-1$ and $\gamma_t$, and we used the mere fact that $\gamma_t$ is a non-torsion unit.

Here we make two remarks about this~\eqref{eq:index} in general.

The verification in~\cite{mau} of~\eqref{eq:index} for~\eqref{1} below is indeed tedious, and for~\eqref{2} is such that the details are omitted. It is to be presumed that for~\eqref{4a} it may test computational limits. But~\eqref{abcde} involves arbitrarily many parameters in the coefficients, and is almost certainly beyond this sort of computation.

In fact if~\eqref{eq:index} fails for infinitely many $t>0$, then some functional unit becomes a fixed perfect power when specialized to infinitely many $t$, while itself not being a perfect power. Then one can apply Siegel's Theorem on integral points and seemingly deduce a contradiction. But alas! this re-introduces ineffectivity.\\

To deal with~\eqref{eq:ge} the height estimate of~\cite[Lemma 2.3]{Be-Sc}, which was one of the starting points of~\cite{Am-Ma-Za}, would have been enough. Even for general cubics~\cite[Lemma 2.3]{Be-Sc} is not enough. But in a special family of cubics, a direct application of~\cite{Am-Ma-Za} gives a more precise result than~\ref{thm:main}. Let $P\in \Z[T,X]$ be an irreducible polynomial, monic of degree $3$ in $X$, such that $P(t,X)$ has only one real zero $\xi_t$ for all sufficiently large positive integers $t$. 
\begin{theorem}
\label{thm:main-cubic}
We suppose that there exists a non-constant unit $\gamma\in\Z[T,\xi]^*$. Then for  any $\e>0$ there are $C(\e),t_0(\e)$, effectively depending on $\e$ and $\xi$, such that for any positive integer $t\geq t_0(\e)$ and any integer $q\neq0$, the solutions $(x,y)\in\Z^2$ of
$$
\Norm(x-\xi_ty)=q
$$
satisfy 
$$
\max\{\vert x\vert,\vert y\vert\}\leq t^{C(\e)}\vert q\vert^{1+\e}.
$$
\end{theorem}

The exponent of $|q|$ here is best possible - see section \ref{sec:irr} for more. The result covers not only $P=X^3-(T^3-1)$ as in \eqref{eq:ge}, but also the interesting example $P=X^3+TX-1$.\\

When the relevant group of units has rank greater than one, a direct use of \cite[Theorem 1.4]{Am-Ma-Za}  is not enough to get things like~\eqref{eq:gea}. A crucial ingredient for the proof of Theorem~\ref{thm:main} is the Theorem~\ref{thm:AMZbis} below.

In order to give a precise statement, we need first some relevant notations. Let $h$ be the Weil height on $\Gm^r(\Qb)\hookrightarrow\P_r(\Qb)$. We fix a projective smooth curve $\Cu$ defined over $\Qb$, with function field $\F=\Qb(\Cu)$. We choose a system of Weil functions associated to a divisor of degree $1$ and a corresponding height on $\Cu(\Qb)$, which we denote by the same letter $h$.  We consider a finitely generated subgroup $\Gamma\subset\Gm^r(\F)$ such that 
\begin{multline}
\label{ass:H2}
\hbox{whenever some quotient of coordinates of some element of $\Gamma$}\\ \hbox{is an algebraic number, it is a root of unity.}
\end{multline}

\begin{theorem}
\label{thm:AMZbis}
We fix $\theta_1,\ldots,\theta_r\in\F^*$. Then there exist $c',C',C''\geq1$ effectively depending only on $\F$, $\Gamma$ and $\theta_1,\ldots,\theta_r$ with the following property. Let $\alphag\in(\Qb^*)^r$ and $\Pt\in\Cu(\Qb)$. If for some $\gammag\in\Gamma$  the value $\gammag_\Pt$ is defined and satisfies
\begin{equation}
\label{eq:hgamma_t}
h(\gammag_\Pt)\geq \max\{C'(1+h(\Pt)),c'\, h(\alphag)\},
\end{equation}
and
\begin{equation}
\label{eq:hyperplane}
\alpha_1\theta_1(\Pt) \gamma_1(\Pt)+\cdots + \alpha_r\theta_r(\Pt)\gamma_r(\Pt)=0
\end{equation}
with no proper vanishing subsums, then the height of $\Pt$ is bounded by $C''$.
\end{theorem}

In Theorems~\ref{thm:main} and~\ref{thm:AMZbis} we have deliberately distinguished between lower case and upper case constants. It seems to us that those in lower case are rather easier to compute explicitly than those in upper case.\\

The plan of this paper is as follows. In section~\ref{sec:ex} we discuss several examples. %In section~\ref{sec:aux} we present two lemmas on heights needed for the proof of Theorem~\ref{thm:main}. 
In section~\ref{sec:aux} we present two lemmas on heights needed for the proofs of the main theorems. In Section~\ref{sec:main-cubic} we derive Theorem~\ref{thm:main-cubic} from the main result of~\cite{Am-Ma-Za}, by the method sketched to treat~\eqref{eq:ge}. Section~\ref{sec:Am-Ma-Za-bis} is devoted to the proof of Theorem~\ref{thm:AMZbis}. Then, in section~\ref{sec:Thomas-bis} we prove Theorem~\ref{thm:main}. After that section~\ref{sec:ex2} provides a detailed discussion of our various examples. And in section~\ref{sec:ss} we focus on ``small solutions" with particular emphasis on the equation $x^3-(t^3-1)y^3=q$. Then in the short section~\ref{sec:irr} we consider some of our results from the closely related viewpoint of diophantine approximation.  Finally in the even shorter appendix \ref{sec:app} we indicate how some details of \cite{Am-Ma-Za2} should be modified to accommodate the changeover from \cite[Assumption 2.1]{Am-Ma-Za2} to Assumption~\ref{ass:Thomas-H}.

As already remarked in \cite{Za} (p.352), one may hope to obtain some sort of uniformity in the coefficients of $P(T,X)$, for example in $x^3-(t^3-a)y^3=q$ for small $a$, generalizing (\ref{eq:ge}).

\section{examples}
\label{sec:ex}
Here are some examples, all for prime $d$ to avoid intermediate fields, except $d=9$. In all these examples the units become immediately apparent on factorising $P+1$ or $P-1$.

For $d=3$ we already considered $P(T,X)=X^3-(T^3-1)$. We may in passing mention also
\begin{equation}\label{XT}
X^3+TX-1,~~~X^3-T^2X-1,~~~X(X-1)(X-T)-1
\end{equation}
which seem not to have been treated for general $q$ (Wakabayashi \cite{wak} does the second for $q=1$, and Lee \cite{lee} and Mignotte-Tzanakis \cite{mat} the third for $q=1$ - see also Thomas \cite{tho} equation (1.5)ii). Thus for example all solutions of 
\begin{equation}\label{cub}
x^3+txy^2-y^3=q
\end{equation}
have 
\begin{equation}\label{cubest}
\max\{|x|,|y|\} \leq t^C|q|^c
\end{equation}
for large $t$. We could not find such a simple result in the literature (let alone the sharp version with any $c>1$ coming from Theorem \ref{thm:main-cubic}). As remarked in \cite{Am-Ma-Za2} (p.902) we can cover small $t$ using linear forms in logarithms.

On the other hand it seems unlikely that one can prove (\ref{cubest}) using only linear forms in logarithms. It is true that the relevant regulator is a lot smaller than usual, but even taking this into account as for example in Bugeaud \cite{bug} Th\'eor\`eme 3 (p.231) gives something like $(3t|q|)^{c\log t}$ in (\ref{cubest}).

It may be worth remembering that no-one knows if for example $x^3-ty^3=1$ entails $\max\{|x|,|y|\} \leq t^C$, even for $t$ large (by well-known results - see for example Bennett~\cite{Bennett2} - there can be at most one solution with $y \neq 0$). Of course we have no suitable functional units, and linear forms in logarithms usually lead only to exponential polynomial dependence. 

On the third hand a well-known conjecture (almost equally well-known to be still open) implies that this holds for any $C > 1$.

For $d=5$ Marzenta \cite{mar} has treated (after earlier work of Lombardo \cite{lom})
\begin{equation}
\label{1}
P(T,X)=X^5+4T^4X-1
\end{equation}
also for $q=1$, and even in four variables with $x_4=0$. The Assumption \ref{ass:Thomas-H} as well as the crucial~\eqref{eq:index} is deduced from a result of Maus with the help of techniques from \cite{Am-Ma-Za2}. So now we can treat general $q$ using Theorem \ref{thm:main} and we no longer need~\eqref{eq:index}. Here $\varrho=2$.

The result of Maus is proved also for
\begin{equation}
\label{2}
P(T,X)=X^5-T^4X-1
\end{equation}
where now $\varrho=3$. Again one can check Assumption \ref{ass:Thomas-H}, and probably also~\eqref{eq:index}; but again the latter is no longer needed. So here too our Theorem \ref{thm:main} applies. Thus for $t$ large we have
$$\max\{|x|,|y|,|z|\} \leq t^C|q|^c$$
for all integers $x,y,z$ with
$$x^5+y^5+z^5-t^4(2x^3z^2-4x^2y^2z+xy^4+yz^4)+t^8xz^4 ~=~q.$$
This may be compared not unfavourably with a celebrated result of Baker \cite{baq} for the equation
\begin{equation}\label{bq}
x^5+(t^5-1)y^5+(t^5-1)^2z^5+5(t^5-1)xyz(xz-y^2)~=~q
\end{equation}
and his explicit estimates
$$\max\{|x|,|y|,|z|\}  \leq t^{2500}q^2$$
for $t > 10^{11}$ (best possible would be $|q|^\lambda$ with any $\lambda>1/3$).

Heuberger \cite{heu} has shown for
\begin{equation}\label{2a}
P(T,X)=X(X^2-1)(X^2-T^2)-1
\end{equation}
that the corresponding binary equation
\begin{equation}\label{hq}
x(x^2-y^2)(x^2-t^2y^2)-y^5=1
\end{equation}
has only the solutions
$$(x,y)~=~(1,0),(0,-1),(1,-1),(-1,-1),(t,-1),(-t,-1)$$
provided $t \geq 36.10^{18}$. He used linear forms in logarithms and was able to calculate the left-hand side of~\eqref{eq:index}. Here $\varrho=4$. Our Theorem \ref{thm:main} applies with general $q$, even in four variables.

Similar conclusions hold for the closely related examples
\begin{multline}
\label{2aa}
X(X^2+1)(X^2-T^2)-1,~~X(X^2-1)(X^2+T^2)-1,\\~~X(X^2+1)(X^2+T^2)-1
\end{multline}
which do not seem to have been considered before.

For $d=7$ we do not know any natural examples in the literature, even with $q=1$. But (\ref{1}) is vaguely connected with the name Aurifeuille, and this led us to
\begin{equation}\label{3}
P(T,X)=X^7+27T^6X-1
\end{equation}
also with $\varrho=3$. Here Assumption \ref{ass:Thomas-H} is not too difficult to check. But now we do not know~\eqref{eq:index}. The calculations of Maus even for (\ref{1}) are substantial,  with one real embedding (and $\varrho_1=1$) and four non-real (and $2\varrho_2=4$), and those for (\ref{2}) are even more complicated with three real and two non-real; and in fact Maus does not give the details because of ``lack of space'' \cite{mau} (p.248).

Nevertheless we can apply Theorem \ref{thm:main} in two variables to
\begin{equation}\label{sep}
x^7+27t^6xy^6-y^7=q
\end{equation}
and in three variables to
\begin{multline*}
x^7+y^7+z^7-7xyz(xz-y^2)^2\\-27t^6(2x^4z^3-9x^3y^2z^2+6x^2y^4z-xy^6-yz^6)+729t^{12}xz^6=q.
\end{multline*}
Here we could in principle do four, five, or even six variables. But the resulting diophantine equations are inconveniently long to display; in the last case there are 343 terms. 

A septic analogue of (\ref{2}) is
\begin{equation}\label{4}
P(T,X)=X^7-T^6X-1. 
\end{equation}
But now $\varrho_1=3$ and $2\varrho_2=4$ so $\varrho=4$ and so it seems even more unlikely that~\eqref{eq:index} can be proved; nevertheless Assumption \ref{ass:Thomas-H} can. So Theorem \ref{thm:main} applies.

We can very probably handle septic analogues of (\ref{2a}),(\ref{2aa}) such as $X(X^4-1)(X^2-T^2)-1$ (with $\varrho=5$) and so on; but we skip over these to discuss higher degree $d$.

Analogues of (\ref{2}) and (\ref{4}) for $d=9,11,13,\ldots$ exist; but it seems that not enough functional units are available as required in part (1) of Assumption \ref{ass:Thomas-H}. The same holds for the ``Baker Quintic'' $X^5-(T^5-1)$ which is behind (\ref{bq}) above. But after (\ref{2a}) we found the nonic
\begin{equation}\label{4a}
P(T,X)=X(X^4-1)(X^4-T^4)-1
\end{equation}
with $\varrho=6$. It turns out that Theorem \ref{thm:main} does apply, even though the degree is not prime. So we can handle eight variables. But the resulting diophantine equation is far too long for Maple 2025 to display.

However to give an idea we calculated the formal norm of 
$$X_0+\xi_tX_1+\xi_t^2X_2+\xi_t^3X_3+\xi_t^4X_4+\xi_t^5X_5+\xi_t^6X_6+\xi_t^7X_7$$
as
\begin{multline*}
N(X_0,X_1,X_2,X_3,X_4,X_5,X_6,X_7)\\ =N_0+N_1t^4+N_2t^8+N_3t^{12}+N_4t^{16}+N_5t^{20}+N_6t^{24}+N_7t^{28}
\end{multline*}
where $N_0=X_0^9+X_1^9+X_2^9+X_3^9+X_4^9+X_5^9+X_6^9+X_7^9+\cdots$ (independent of $t$) up to terms of degree less than 9 in each variable, and $N_7=-X_0X_7^4LL'Q$ with
$$L=\sum_{i=0}^7X_i,~~L'=\sum_{i=0}^7(-1)^iX_i,~~Q=\sum_{j=0}^7\sum_{k=0}^7(-1)^{(k-j)/2}X_jX_k,$$
the double sum being restricted to $j,k$ of the same parity.

Thus we have

\begin{proposition} There is effective $t_0$ such that if $t \geq t_0$ then all solutions of
$$N(x_0,x_1,x_2,x_3,x_4,x_5,x_6,x_7)=q$$
in integers $x_0,x_1,x_2,x_3,x_4,x_5,x_6,x_7$ satisfy
\begin{equation}\label{nonest}
\max\{|x_0|,|x_1|,|x_2|,|x_3|,|x_4|,|x_5|,|x_6|,|x_7|\} \leq t^C|q|^c
\end{equation}
for some effective $C,c$.
\end{proposition}

Most of these examples $P(T,X)$ are special cases of polynomials defining the Ankeny-Brauer-Chowla ``ABC" fields mentioned in \cite{Am-Ma-Za2} (Theorem \ref{ass:Thomas-H} and p.903). These have the shape
\begin{equation}\label{ABC}
(X-A_1)\cdots(X-A_m)(X^2+B_1X+C_1)\cdots(X^2+B_nX+C_n)+1
\end{equation}
with $A_1,\ldots,A_m,B_1,C_1,\ldots,B_n,C_n$ in ${\Z}[T]$ such that $B_i^2-4C_i<0$ at all sufficiently large $T=t>0$. One can also change the sign of the last term but that is no more general so we stick with (\ref{ABC}) in order to agree with \cite{Am-Ma-Za2}. The fields are constructed to have a full set of functional units, so they are good candidates for Theorem \ref{thm:main}, at least if $m$ is odd. But it may not be easy to rule out intermediate fields or verify Assumption \ref{ass:Thomas-H}.

In \cite{Am-Ma-Za2} we treated $m=3,n=0$ with $q=1$. But Theorem \ref{thm:main} applies with general $q$, so to the equation
\begin{equation}\label{abc}
(x-A(t)y)(x-B(t)y)(x-C(t)y)+y^3~=~q
\end{equation}
for distinct $A,B,C$ in ${\Z}[T]$. It may possibly apply to
\begin{equation}\label{abcde}
(x-A(t)y)(x-B(t)y)(x-C(t)y)(x-D(t)y)(x-E(t)y)+y^5~=~q
\end{equation}
(even with $z$ thrown in), but so far we have not been able to check Assumption \ref{ass:Thomas-H} part 1.

We have referred to linear forms in logarithms regarding (\ref{cub}) and (\ref{hq}); and indeed these can probably be used similarly to handle Thue equations effectively as in (\ref{sep}). As already mentioned, our functional units mean that the relevant regulators do not explode as expected (see \cite{Am-Ma-Za2} p.910). Crucial seems to be the fact that only two variables appear (in general the space $W$ in Theorem \ref{thm:main} should have dimension at most 2).

As Marzenta \cite{mar} (p.289) points out, there is some work in the literature which leads to the possibility of effectively treating more than two variables using linear forms in logarithms. But in any case none of this work covers six variables. Theorem 2 of Bajpai \cite{baj} (p.1272) can handle five variables, but is restricted to even degree $d$ (with signature conditions). And Bajpai-Bennett \cite{bab} Theorem 1 (p.422) is restricted to $d \leq 6$ (also with signature conditions), while Gy\"ory \cite{gyo} Theorem 6 (p.164) also needs even $d$ (because of a CM field condition). Also Gy\"ory-Lov\'asz \cite{gal} Theorem 1 (p.174) is restricted to three variables (also with a CM condition).

In any case we already remarked that linear forms in logarithms do not seem to give polynomial bounds like (\ref{cubest}) or (\ref{nonest}).

\section{Auxiliary Results}
\label{sec:aux}
%\subsection{Heights and Geometry of Numbers}
%\label{subs:heights}
Given $\alpha_1,\ldots,\alpha_d\in\Qb$ we denote by $h(\alphag)$ the normalized, logarithmic Weil height of $\alphag=(\alpha_1,\ldots,\alpha_d)\in\Qb^d$, which we identify with the corresponding point in $\P_{d-1}(\Qb)$. For an algebraic number $\alpha$ we put $h(\alpha)=h(1,\alpha)$. We also adopt the following:
\begin{convention}
\label{convention}
Let $k$ be a fixed number field of degree $d$. We fix an order $\sigma_1,\ldots,\sigma_d$ of the $\Q$-immersions of $k$ in $\C$.
For $\alpha\in k$ we then denote $\alpha_j=\sigma_j(\alpha)$ and $\alphag=(\alpha_1,\ldots,\alpha_d)\in\P_{d-1}(\Qb)$.
\end{convention}

The following result is more or less implicit in the literature (see for example Baker~\cite{bac}  p.188).
\begin{lemma}
\label{lem:box}
Let $k$ be a number field of degree $d$ with group of units $\O_k^*$ of rank $\varrho$. Let $u_1,\ldots,u_\varrho$ be a system of multiplicatively independent units of~$\O_k^*$. Then for any $\alpha\in\O_K$ there exists a unit $\eta$ in the subgroup generated by $u_1,\ldots,u_\varrho$ such that 
$$
h(\etag^{-1}\alphag)\leq\tfrac1d\log\vert N^k_\Q(\alpha)\vert+h(\ug_1)+\cdots+h(\ug_\varrho).
$$
%where for $\theta\in k$ we have denoted $\thetag=(\sigma_1\theta,\cdots,\sigma_d\theta)$, with $\sigma_1,\ldots,\sigma_d$ the $\Q$-immersions of $k$ in $\C$.
\end{lemma}
\begin{proof}
We fix an arbitrary order on the set of $\varrho+1$ archimedean places $\{v\mid\infty\}$ of $k$. This allows us to use this set as a set of indexes. Instead of the usual logarithmic immersion $\L\;\colon K^*\rightarrow\R^{\varrho+1}$ defined by $\L(\theta)=(d_v\log\vert\theta\vert_v)_{v\mid\infty}$, with $d_v=1$ if $v$ is real and $d_v=2$ otherwise, we take here 
$$
\L(\theta)=(\log\vert\theta\vert_v)_{v\mid\infty}.
$$
Thus for $\theta\in\O_k$ we have $h(\thetag)\leq\Vert\L(\theta)\Vert$, where $\Vert\star\Vert$ denotes the sup norm. 
%Thus $\L(\O_k^*)\subseteq H$ with $H$ the hyperplan $\{\xg\in\R^{r+1},\; \sum_v d_vx_v=0\}$. 
By assumption $\L(u_1),\ldots,\L(u_\varrho)$ generate a lattice of the hyperplan $\{\xg\in\R^{\varrho+1},\; \sum_v d_vx_v=0\}$, with fundamental parallelogram
$$
P=\{\lambda_1\L(u_1)+\cdots+\lambda_\varrho\L(u_\varrho)\;\vert\;
0\leq\lambda_1,\ldots,\lambda_\varrho<1\}.
$$
Let $\pi\colon \R^{\varrho+1}\rightarrow H$ be the projection $\pi(\xg)=\xg-\tfrac1d\sum_{v\mid\infty} d_vx_v\cdot\1$ with $\1=(1,\ldots,1)$. Thus there exists a unit $\eta$ in the subgroup generated by $u_1,\ldots,u_\varrho$ such that $(\pi\circ\L)(\eta^{-1}\alpha)\in P$. Note that for $\theta\in k^*$ we have 
$$
\L(\theta)=\tfrac1d\log\vert N^k_\Q(\theta)\vert\cdot\1+(\pi\circ\L)(\theta).
$$
Note also that $\log\vert N^k_\Q(\eta^{-1}\alpha)\vert=\log\vert N^k_\Q(\alpha)\vert\geq 0$. By the above, 
\begin{align*}
h(\etag^{-1}\alphag)\leq\Vert\L(\eta^{-1}\alpha)\Vert_\infty
&\leq\tfrac1d\vert\log\big\vert N^k_\Q(\eta^{-1}\alpha)\vert\big\vert+\Vert (\pi\circ\L)(\eta^{-1}\alpha)\Vert_\infty\\
&\leq\tfrac1d\log\vert N^k_\Q(\alpha)\vert+h(\ug_1)+\cdots+h(\ug_\varrho).
\end{align*} 
\end{proof}

%\subsection{On the height of the conjugates.} 
Let $\alpha\in\C$ be an algebraic number of degree $d\geq2$ with set of conjugates $\Omega=\{\alpha_1,\ldots,\alpha_d\}$. In some context, it could be useful to compare the height of the projective point $\alphag=(\alpha_1,\ldots,\alpha_d)$ with $h(A)$, the height of a subset $A\subseteq\Omega$ viewed as a subset of the projective space $\P_{\vert A\vert-1}(\Qb)$. We remark that $h(A)$ does not depend on the ordering. For a singleton $A$ we adopt the usual convention $h(A)=0$. 

Of course $h(A)\leq h(\alphag)$. We look for inequalities in the opposite direction, which cannot hold unconditionally. Let $G$ be the Galois group of $\Q(\alpha_1,\ldots,\alpha_d)/\Q$ which we identify to a transitive group of permutations on~$\Omega$. 

We recall (see~\cite[\S 1.5]{Di-Mo}) that a non-empty subset $A\subseteq\Omega$ is a block for $G$ if for each $\sigma\in G$ either $\sigma(A)=A$ or $\sigma(A)\cap A=\emptyset$. Thus $\{\alpha\}$ is a block, and the whole set $\Omega$ is also a block; they are called trivial blocks.
\begin{lemma}
\label{lem:include}
Let $A_0$ be a subset of $\Omega$. Then there exists a block $A\supseteq A_0$ such that 
Then 
$$
h(A)\leq 2^{\vert A\vert-\vert A_0\vert}h(A_0).
$$
\end{lemma}
\begin{proof}
Let 
$$
\Lambda=\{A\;\vert\; A_0\subseteq A\subseteq\Omega \hbox{ and } h(A)\leq 2^{\vert A\vert-\vert A_0\vert}h(A_0)\}.
$$
Then $A_0\in\Lambda\neq\emptyset$. Let $A$ be a maximal element of $\Lambda$ with respect to the inclusion. We assume by contradiction that $A$ is not a block. Thus there exists $\sigma\in G$ such that $\sigma(A)\neq A$ and $\sigma(A)\cap A\neq\emptyset$. Hence we can 
find $\beta,\gamma\in A$ such that $\sigma(\beta)\not\in A$ and $\sigma(\gamma)\in A$. Let $A'=A\cup\{\sigma(\beta)\}\supsetneq A$. 

At this stage we remark that
\begin{multline*}
h(\theta_1,\ldots,\theta_{d-1},\theta_d)=h(\theta_1/\theta_{d-1},\ldots,\theta_{d-2}/\theta_{d-1},1,\theta_d/\theta_{d-1})\\
\leq h(\theta_1/\theta_{d-1},\ldots,\theta_{d-2}/\theta_{d-1},1)+h(\theta_d/\theta_{d-1},1)
=h(\theta_1,\ldots,\theta_{d-1})+h(\theta_d/\theta_{d-1})
\end{multline*}
for any algebraic numbers $\theta_1,\ldots,\theta_d$.
Thus
\begin{multline*}
h({A'})\leq h({A})+h(\sigma(\beta)/\sigma(\gamma))=h(A)+h(\beta/\gamma)\\
\leq 2h(A)\leq 2^{1+\vert A\vert-\vert A_0\vert}h({A_0})=2^{\vert A'\vert-\vert A_0\vert}h({A_0}),
\end{multline*}
which contradicts the maximality of $A$.
\end{proof}
We say that $G$ is primitive if $G$ has no nontrivial blocks. We know (\cite[Corollary 1.5A]{Di-Mo}) that $G$ is primitive if and only if the stabiliser of $\{\alpha\}$ is a maximal subgroup of $G$. In turn, this last claim is equivalent by Galois Theory to the fact that $\Q(\alpha)/\Q$ is primitive, \ie that it does not contain non-trivial subextensions. Thus Lemma~\ref{lem:include} implies:
\begin{corollary}
\label{cor:height}
Let $\alpha$ be an algebraic number of degree $d\geq2$. Let us assume $\Q(\alpha)/\Q$ is primitive. Then for any conjugate $\beta$ of $\alpha$ with $\beta\neq\alpha$, we have $h(\alphag)\leq 2^{d-2}h(\beta/\alpha)$.
\end{corollary}
\begin{proof}
Apply Lemma~\ref{lem:include} with $A_0=\{\alpha,\beta\}$.
\end{proof}

\section{Proof of Theorem~\ref{thm:main-cubic}} 
\label{sec:main-cubic}
Put for short $\gammag=(\gamma,\gamma',\gamma'')$ with the other two embeddings, and write $\alpha=x-\xi_ty$. By Lemma~\ref{lem:box} we can find an integer $n$ such that $\beta=\gamma_t^{-n}\alpha$ satisfies $h(\beta,\beta',\beta'')\leq \tfrac13\log |q| +h(\gammag_t)$. Replacing $\gamma$ by $\gamma^{-1}$ 
and $h(\gammag_t)$ by $h(\gammag_t^{-1})\leq 2h(\gammag_t)$ if needed, we can assume $n\geq 0$ and 
\begin{equation}
\label{eq:beta}
h(\beta,\beta',\beta'')\leq \frac13\log |q| +2h(\gammag_t).
\end{equation}
By standard estimates (assuming $t \geq 2$)
\begin{equation}
\label{eq:xi}
\vert{\rm Im}(\xi'_t)\vert \geq t^{-c},\qquad h(\xi_t)\leq c\log t
\end{equation}
and
\begin{equation}
\label{eq:gamma}
c^{-1}\log t\leq h(\gammag_t)\leq c\log t.
\end{equation}
for a positive constant $c>1$ (which does not depend on~$t$). 

Since $\beta\gamma_t^n=\alpha=x-\xi_ty\in\Q+\Q\xi_t$, the point $(\beta \gamma_t^n,\beta' {\gamma'}_t^n,\beta''\,{\xi''_t}^n)$ belongs to the line in $\P_2$ through the points $(1,1,1)$ and $(\xi_t,\xi_t',\xi''_t)$. This line has equation
$$
\left\vert\begin{matrix}
x_1&x_2&x_3\\
1&1&1\\
\xi_t&\xi_t'&\xi''_t
\end{matrix}\right\vert=(\xi''_t-\xi_t')x_1+(\xi_t-\xi''_t)x_2+(\xi_t'-\xi_t)x_3=0.
$$
Thus
\begin{equation}
\label{eq:zerosum}
(\xi''_t-\xi_t')\beta \gamma_t^n+(\xi_t-\xi''_t)\beta' {\gamma'}_t^n+(\xi_t'-\xi_t)\beta''\,{\gamma''_t}^n=0.
\end{equation}
It is of course the usual Siegel Identity. Clearly we do not have proper vanishing subsums in this equation. We want to apply~\cite[Theorem 4.1]{Am-Ma-Za}. To do that, we still need $n\geq K$ with $K$ larger than a fixed constant.

Recall that $n\geq0$. By elementary height estimates and by~\eqref{eq:beta}
\begin{multline*}
h(\alpha,\alpha',\alpha'')
=h(\beta \gamma_t^n,\beta' {\gamma'}_t^n,\beta''\,{\gamma''_t}^n)\\
\leq nh(\gammag_t)+h(\beta,\beta',\beta'')
\leq (n+2)h(\gammag_t)+\frac13\log \vert q\vert.
\end{multline*}
By~\eqref{eq:xi} we have $\vert\alpha'\vert\geq \vert {\rm Im}(\xi'_t)\vert\cdot \vert y\vert\geq t^{-c}\vert y\vert$. We assume 
$$
\vert y\vert\geq t^{c(K+3)}\vert q\vert^{1+\e}.
$$
Then, by the upper bound for $h(\gammag_t)$ in~\eqref{eq:gamma},
\begin{align*}
h(\alpha,\alpha',\alpha'')\geq\log\vert\alpha'\vert
&\geq c(K+2)\log t+(1+\e)\log \vert q\vert\\
&\geq(K+2)h(\gammag_t)+(1+\e)\log \vert q\vert.
\end{align*}
Comparing the upper bound and the lower bound for $h(\alpha,\alpha',\alpha'')$ we get
$$
(K+2)h(\gammag_t)+(\tfrac23+\e)\log\vert q\vert\leq(n+2)h(\gammag_t).
$$
Thus
\begin{equation}
\label{eq:nn}
n\geq(\tfrac23+\e)h(\gammag_t)^{-1}\log\vert q\vert+K.
\end{equation}

In particular $n\geq K$ and, if $K$ is sufficiently large with respect to $\xi$ and $\gamma$, we can apply~\cite[Theorem 4.1]{Am-Ma-Za} with $r=3$ and with $(f_1,f_2,f_3)=(\gamma,\gamma',\gamma'')$ to the equation~\eqref{eq:zerosum}. By that theorem, and by taking into account \cite[Remark 4.2~i) and ii)]{Am-Ma-Za}, 
\begin{multline*}
\frac{1}{2}h(\gammag_t)
\leq\frac{1}{n}h(\xi''_t-\xi_t')\beta,(\xi_t-\xi''_t)\beta',(\xi_t'-\xi_t)\beta'')\\
+O(\tfrac{1}{K}h(\gammag_t)+h(\gammag_t)^{1/2}+K),
\end{multline*}
where the implicit constant in the big-O depends only on $P$, as also below. By~\eqref{eq:beta},~\eqref{eq:xi} and~\eqref{eq:gamma} we have
\begin{align*}
h(\xi''_t-\xi_t')\beta,(\xi_t-\xi''_t)\beta',&(\xi_t'-\xi_t)\beta'')\\
&\leq h(\xi''_t-\xi_t',\xi_t-\xi''_t,\xi_t'-\xi_t)+h(\beta,\beta',\beta'')\\
&\leq  \tfrac13\log\vert q\vert+O(h(\gammag_t)+\log t)\\
&\leq\tfrac13\log\vert q\vert+O(h(\gammag_t)).
\end{align*}
Thus
$$
h(\gammag_t)\leq\tfrac1n\cdot\tfrac23\log\vert q\vert+ O(\tfrac{1}{K}h(\gammag_t)+\tfrac{1}{n}h(\gammag_t)+h(\gammag_t)^{1/2}+K).
$$
By~\eqref{eq:nn} again
$$
\tfrac1n\cdot\tfrac23\log\vert q\vert\leq(1+\tfrac32\e)^{-1}h(\gammag_t).
$$
Hence if $K>K(\e)$ we have $h(\gammag_t) \leq c_0(\e)$ and thus $t\leq t_0(\e)$ by~\eqref{eq:gamma}. This gives the desired upper bound for $|y|$; and the same for $|x|$ follows easily from $P(t,x/y)=q/y^3$.

\section{Proof of Theorem~\ref{thm:AMZbis}} 
\label{sec:Am-Ma-Za-bis}

Theorem~\ref{thm:AMZbis} is a variant of~\cite[Proposition 6.1]{Am-Ma-Za} where we make the result explicit with respect to the the height of $\alpha_1,\ldots,\alpha_r$. The quoted proposition was in turn a consequence of \cite[Theorem 1.4]{Am-Ma-Za}, which we state below for the reader's convenience.
\begin{theorem}
\label{thm:AMZ-Thm1.4}
Let $r\geq 2$ and $f_1,\ldots,f_r\in\F$ be non-zero rational functions such that $f_i/f_j$ is non-constant for some $i$ and $j$. Then there exists a positive real number $C$ depending only on $f_1,\ldots,f_r$, having the following properties. Let $\alphag=(\alpha_1:\cdots:\alpha_r)\in\P^{r-1}(\Qb)$.  
Consider, for a natural number $n$, a solution $P\in \Cu(\Qb)$ of the equation
$$
\alpha_1 f_1(\Pt)^n+\cdots + \alpha_r f_r(\Pt)^n =0.
$$
Then, if $n\geq C$ and if there are no proper vanishing subsums, we have 
$$
h(\Pt)\leq \frac{r h(\alphag)}{n}+C.
$$
\end{theorem}

We deduce Theorem~\ref{thm:AMZbis} from this last theorem as we did in~\cite{Am-Ma-Za} to prove Proposition 6.1. We first recall a version of the classical functorial bound for the height.\\

Let $\Cu$ be a projective smooth curve defined over $\Qb$, with function field denoted $\F:=\Qb(\Cu)$. Given $f_1,\ldots,f_r\in\F$ not all zero, we put
$$
\div(f_1,\ldots,f_r):=\sum_P \min_j\ord_P(f_j) P.
$$ 
%As a special case, let $f$ be a non-zero rational function on $\Cu$. We define as usual  its degree $d(f)$ as the the degree of the polar divisor $\div(f)_{\infty}=-\div(1,f)$. This is the geometric height of $(1:f)\in\P^1(\F)$. 

%We define an arithmetic height $h(\cdot)$ of a rational function $f$ on $\Cu$ as follows. We choose once and for all a non-constant $t\in\F\backslash\Qb$. Let $F(X,Y)\in\Qb[X,Y]$ be the irreducible polynomial such that $F(t,f)=0$ (note that $F$ has degree at most $d(f)$ in $X$ and at most $d(t)$ in $Y$). 
%\begin{definition} 
%\label{def.height}  
%For a function $f\in\F$, we define the height $h(f)$ as the projective Weil height of the vector of the coefficients of $F$.
%\end{definition}
%Clearly $h(1/f)=h(f)$. Also, this coincides with the affine height on $\Qb[t]$ (if $\Cu$ is the affine line and $f=P(t)$ is a polynomial then $F(X,Y)=P(X)-Y$.)\\
%

\begin{lemma}[\cite{Am-Ma-Za}, Lemma 3.3]
\label{lem:AMZ-Lem3.3}
For $r \geq 2$ let $f_1,\ldots,f_r\in\F$ and $\Pt\in\Cu(\Qb)$, not a pole or a common zero of $f_1,\ldots,f_r$. Put $d:=-\deg\div(f_1,\ldots,f_r)$. Then
$$
h(f_1(\Pt),\ldots,f_r(\Pt)) = d h(\Pt) + O(1+h(\Pt)^{1/2})
$$
where the implicit constant in the big-$O$ may depend on $f_1,\ldots,f_r$ but not on $\Pt$.
\end{lemma}

\begin{proof}[Proof of Theorem~\ref{thm:AMZbis}] We closely follow the proof of Proposition 6.1 in~\cite{Am-Ma-Za}, pp. 2639-40. For the convenience of the reader we repeat all the relevant steps. Dividing~\eqref{eq:hyperplane} by $\gamma_1(\Pt)$ and replacing $\Gamma$ by 
$$
\{(1,f_2/f_1,\ldots,f_r/f_1)\;\vert\; \fb\in\Gamma\}, 
$$
we may assume $\gamma_1=1$ and $\Gamma\subseteq \{x_1=1\}$. We decompose $\Gamma$ as $\Gamma_{\rm tors}\oplus\Gamma'$, where $\Gamma'$ is freely generated. We fix free generators $\gb_1,\ldots,\gb_\kappa$ of $\Gamma'$ and we write
$$
\gammag=\omegag\gb_1^{\lambda_1}\cdots\gb_\kappa^{\lambda_\kappa}.
$$
with $\omegag\in\Gamma_{\rm tors}$ and $\lambda_1,\ldots,\lambda_\kappa\in\Z$. Let $Q>1$ be an integer which will be fixed later in~\eqref{eq:Q}. Let $A=\tfrac12Q\max\vert\lambda_j\vert$. By Dirichlet's Theorem on simultaneous approximation, there exists a positive integer $q< Q^{\kappa}$ and integers $p_j$ such that 
$$
\left\vert q\frac{\lambda_j}{A}-p_j\right\vert\leq \frac{1}{Q}~~~(j=1,\ldots,k).
$$
Let $n=\lceil A/q\rceil$. We write $\lambda_j=np_j+r_j$ and we set, for $i=1,\ldots,r$, 
\begin{equation}
\label{eq:reparametrize}
\rhog= \prod_{j=1}^\kappa \gb_j^{r_j}\in\Gamma,
\qquad \fb= \prod_{j=1}^\kappa \gb_j^{p_j}\in\Gamma,
\qquad \alpha'_i=\alpha_i\omega_i\theta_i(P)\rho_i(P).
\end{equation}
Since $\gamma_i=\omega_i\rho_if_i^n$, equation~\eqref{eq:hyperplane} can be rewritten as 
\begin{equation}
\label{eq:AMZ}
\alpha'_1f_1(\Pt)^n+\cdots+\alpha'_rf_r(\Pt)^n=0.
\end{equation}
By assumption we do not have proper vanishing subsums. We have
$$
\vert p_j\vert\leq \left\vert q\frac{\lambda_j}{A}\right\vert+ \frac{1}{Q}\leq (2q+1)Q^{-1}\leq 2Q^{\kappa-1}.
$$
Thus, $\fb$ belongs to a finite set, depending only on $\Gamma$ and on $Q$, as the exponents $\lambda_j$ vary. Moreover
$$
\vert r_j\vert=\left\vert\frac{A}{q}\left(q\frac{\lambda_j}{A}-p_j\right)-\left(n-\frac{A}{q}\right)p_j\right\vert
\leq nQ^{-1}+2Q^{\kappa-1}.
$$
Hence $-\deg\div(1,\theta_i\rho_i)= O(n/Q+Q^{\kappa-1})$, where the implicit constant in the big-$O$ depends only on $\theta_1,\ldots,\theta_r$ and on $\Gamma$. Using (only the upper bound in) Lemma~\ref{lem:AMZ-Lem3.3} for the two functions $1$ and $\theta_i\rho_i$, we see that
$$
h(\theta_1(P)\rho_1(P),\ldots,\theta_r(P)\rho_r(P))\leq\sum_ih(\theta_i(P)\rho_i(P))\leq c_2(n/Q+Q^{\kappa-1})h(\Pt),
$$
where the constant $c_2$ depends only on $\theta_1,\ldots,\theta_r$ and on $\Gamma$. 
Hence, by definition~\eqref{eq:reparametrize} of $\alphag'$,
\begin{equation}
\label{eq:alpha'}
h(\alphag')\leq c_2(n/Q+Q^{\kappa-1})h(\Pt)+h(\alphag).
\end{equation}
We want to apply Theorem~\ref{thm:AMZ-Thm1.4} to equation~\eqref{eq:AMZ}. Since $\fb$ belongs to a finite set depending only on $\Gamma$ and on $Q$, 
the constant $C=C_Q>0$ appearing in this theorem depends not only on $\theta_1,\ldots,\theta_r,\Gamma$ but also on $Q$. 

There is a $j_0$ such that $\lambda_{j_0}=\pm 2AQ^{-1}$ and thus $p_{j_0}\neq0$, otherwise
$$
\frac{2}{Q}\leq\frac{2q}{Q}=\left\vert q\frac{\lambda_{j_0}}{A}\right\vert=\left\vert q\frac{\lambda_{j_0}}{A}-p_{j_0}\right\vert< \frac{1}{Q},
$$
a contradiction. Since $\gb_1,\ldots,\gb_\kappa$ is a basis of $\Gamma'$ we thus have $\fb\not\in\Gamma_{\rm tors}$. Since $\Gamma\subseteq \{x_1=1\}$, we have $f_1=1$. Thus, by Assumption~\eqref{ass:H2}, $f_i=f_i/f_1$ is non-constant for some $i>1$, which is one of the assumptions of~Theorem~\ref{thm:AMZ-Thm1.4}.

We now need to be more precise with respect to op.cit., where the unboundedness of $A$ was enough to conclude. Again by Lemma~\ref{lem:AMZ-Lem3.3} 
 (with $\fb$ replaced by $\gb_j$)
\begin{equation}
\label{eq:A}
h(\gammag_\Pt)\leq\sum_{j=1}^\kappa\vert\lambda_j\vert h(\gb_j(\Pt))\leq c_3AQ^{-1} (1+h(\Pt)),
\end{equation}
where the constant $c_3$ depends only on $\Gamma$. Recall that $A/q\leq n$ for some integer $q$ with $1\leq q\leq Q^\kappa$. Thus, by~\eqref{eq:A}, 
\begin{equation}
\label{eq:n}
n\geq \frac{A}{q}\geq \frac{A}{Q^\kappa}
\geq\frac{h(\gammag_\Pt)}{c_3Q^{\kappa-1}(1+h(\Pt))}.
\end{equation}
By assumption~\eqref{eq:hgamma_t} we have $h(\gammag_\Pt)\geq C'(1+h(\Pt))$. Hence
\begin{equation}
\label{eq:n2}
n\geq C'c_3^{-1}Q^{-(\kappa-1)}. 
\end{equation}
We assume
\begin{equation}
\label{eq:C'-1}
C'\geq c_3C_QQ^{\kappa-1}.
\end{equation}
Then $n\geq C_Q$ and we can apply Theorem~\ref{thm:AMZ-Thm1.4} to~\eqref{eq:AMZ}. By this theorem, taking into account the bound~\eqref{eq:alpha'} for the height of $\alphag'$, we get 
\begin{equation}
\label{eq:hP}
h(\Pt)\leq \frac{r h(\alphag')}{n}+C_Q\leq rc_2\left(\frac1{Q}+\frac{Q^{\kappa-1}}{n}\right)h(\Pt)+\frac{rh(\alphag)}{n}+C_Q.
\end{equation}
We also assume
\begin{equation}
\label{eq:C'-2}
C'\geq c_3Q^{2\kappa-1}
\end{equation}
so that $n\geq Q^{\kappa}$ by~\eqref{eq:n2}. Choosing
\begin{equation}
\label{eq:Q}
Q=\lceil 8rc_2\rceil
\end{equation}
we get 
\begin{equation}
\label{eq:1}
rc_2\left(\frac1{Q}+\frac{Q^{\kappa-1}}{n}\right)\leq2rc_2Q^{-1}\leq \frac14.
\end{equation}
By assumption~\eqref{eq:hgamma_t} we also have $h(\gammag_\Pt)\geq c'\, h(\alphag)$. We choose
\begin{equation}
\label{eq:c'}
c'=4rc_3Q^{\kappa-1}.
\end{equation}
Then, by~\eqref{eq:n},
\begin{equation}
\label{eq:2}
\frac{rh(\alphag)}{n}\leq rc'^{-1}c_3Q^{\kappa-1}(1+h(\Pt))=\frac14(1+h(\Pt)).
\end{equation}
Replacing~\eqref{eq:1} and~\eqref{eq:2} into~\eqref{eq:hP} we get $h(\Pt)\leq \tfrac12h(\Pt)+\tfrac14+C_Q$ and 
$$
h(\Pt)\leq \frac12+2C_Q.
$$
By~\eqref{eq:Q},~\eqref{eq:C'-1},~\eqref{eq:C'-2} and~\eqref{eq:c'}, the theorem is proved with
$$
Q=\lceil 8rc_2\rceil,\quad C'=c_3Q^{\kappa-1}\max\{C_Q,c_3Q^{\kappa}\},\quad c'=4rc_3Q^{\kappa-1},
$$
and $C''=1/2+2C_Q$.
\end{proof}

\section{Proof of Theorem~\ref{thm:main}}
\label{sec:Thomas-bis}

Let us recall from the introduction the relevant notations. Let $P\in \Z[T,X]$ be an irreducible polynomial, monic of degree $d$ in $X$, and let $\xi\in\overline{\Q(T)}$ be one of its roots. By~\cite[Lemma 3.1]{Am-Ma-Za2}, for $t\in\N$ large $[\Q(\xi_t):\Q]=d$ and moreover the rank of the units group of $\Q(\xi_t)$ is constant (see appendix). %For $t\in\N$ large, $[\Q(\xi_t):\Q]=d$ and the numbers $r_1$ and $2r_2$ of real and imaginary immersions of the number field $\Q(\xi_t)$ are both constants by~\cite[Lemma 3.1]{Am-Ma-Za2}. As usual, we let $r=r_1+r_2-1$.\\
We need a preliminary result from~\cite{Am-Ma-Za2}~(see appendix).

\begin{lemma}[part of \cite{Am-Ma-Za2}, Lemma 3.3]
\label{lem:AMZ2-Lem3.3}
Let us assume~\ref{ass:Thomas-H}. Then there exists $c_4\geq1$ effectively depending on $\xi$ such that for an integer $t\geq c_4$ the following holds:
\begin{enumerate}[(1)]
%\item $\rank\Z[\xi_t]^*=\rank (\Z[T,\xi]^*)_t=\rank \Z[T,\xi]^*=r$.
\item $\rank\Z[\xi_t]^*=\rank (\Z[T,\xi]^*)_t=\rank \Z[T,\xi]^*=\rank \O_{\Q(\xi_t)}^*$.
\item $\Q(\xi)_t=\Q(\xi_t)$ and $[\Q(\xi)_t:\Q]=d$. 
\end{enumerate}
\end{lemma}

We also need a complement to \cite[Proposition 5.4]{Am-Ma-Za2}, the statement of which we first recall for the reader's convenience~(see appendix).

\begin{proposition}[\cite{Am-Ma-Za2}, Proposition 5.4]
\label{prop:AMZ2-Prop5.4}
Let us assume~\ref{ass:Thomas-H}(2). Then there exists $c_5\geq1$ effectively depending on $\xi$ such that for an integer $t\geq c_5$ and for $\mu\in\Z[T,\xi]^*$ we have $\Q(T,\mu)_t=\Q(\mu_t)$.
\end{proposition}

\begin{proposition}
\label{prop:Qmu+}
Let us assume~\ref{ass:Thomas-H}(2) and let us suppose $\Q(T,\xi)/\Q(T)$ be primitive  (\ie that it does not contain non trivial subextensions). Then there exists $c_6\geq1$ effectively depending on $\xi$ such that for an integer $t\geq c_6$ and for $\mu\in\Z[T,\xi]^*$ with $\mu_t\neq\pm1$, we have $\Q(\mu_t)=\Q(\xi_t)$. Moreover, the only possible non trivial subextensions of $\Q(\xi_t)/\Q$ are quadratic imaginary.
\end{proposition}
\begin{proof} 
We choose $c_6=\max(c_4,c_5)$ where $c_4$ and $c_5$ are in the statements of  Lemma~\ref{lem:AMZ2-Lem3.3} and Proposition~\ref{prop:AMZ2-Prop5.4} respectively. Let $t\geq c_6$.

By Proposition~\ref{prop:AMZ2-Prop5.4}, $\Q(T,\mu)_t=\Q(\mu_t)$. Since $\Q(T,\xi)/\Q(T)$ is primitive, either $\Q(T,\mu)=\Q(T)$ or $\Q(T,\mu)=\Q(T,\xi)$. In the first case, $\Q(\mu_t)=\Q$ and thus $\mu_t=\pm1$ as a rational unity, which is excluded. Thus $\Q(T,\mu)=\Q(T,\xi)$ and $\Q(\mu_t)=\Q(T,\mu)_t=\Q(T,\xi)_t=\Q(\xi_t)$, by Proposition~\ref{prop:AMZ2-Prop5.4} and by Lemma~\ref{lem:AMZ2-Lem3.3}(2). This shows the first claim. 

To show the second claim, let $E/\Q$ be a subextension of $\Q(\xi_t)/\Q$ with $E\neq\Q$ and $E/\Q$ not imaginary quadratic. Then we can find a unit $v\in E$ which is not a root of unity. Since $\Z[T,\xi]_t^*$ has the same rank as the rank of the unit group of $\Q(\xi_t)$ (see Lemma~\ref{lem:AMZ2-Lem3.3}(1)), a power $v^l$ is in $\Z[T,\xi]_t^*$. Note that $v^l\not\in\Q$, since otherwise $v$ would be a root of unity. Summing up, we find $u\in\Z[T,\xi]^*$ with $u_t=v^l\in E\backslash\Q$. By Proposition~\ref{prop:AMZ2-Prop5.4} again, $\Q(u_t)=\Q(T,u)_t$ and $\Q(T,u)=\Q(T,\xi)$ since $\Q(T,\xi)/\Q(T)$ is primitive and $u\not\in\Q(T)$. Thus $\Q(\xi_t)\supseteq E\supseteq\Q(u_t)=\Q(T,\xi)_t=\Q(\xi_t)$.
%
%Let first assume $\Q(u_t)=E$. %By the same argument at the beginning of the proof, either $E=\Q(\xi_t)$ or $E=\Q$
%By the same argument at the beginning of the proof $E=\Q(T,u)_t$ and $\Q(T,u)=\Q(T)$ or $\Q(T,u)=\Q(T,\xi)$ since $\Q(T,\xi)/\Q(T)$ is primitive. Thus either $E=\Q(T)_t=\Q$ or $E=\Q(T,\xi)_t=\Q(\xi_t)$, where the last equality comes from~\cite[Lemma 3.3(2)]{Am-Ma-Za2}).
%
%Let now assume that $E$ is a CM field and $\Q(u_t)$ is its totally real subfield. Again, $\Q(u_t)=\Q(T,u)_t$ and $\Q(T,u)=\Q(T)$ or $\Q(T,u)=\Q(T,\xi)$. In the first 
%case $\Q(u_t)=\Q$ an thus $E/\Q$ is an imaginary quadratic extension. In the latter, $E$ is a quadratic extension of $\Q(u_t)=\Q(\xi_t)$, a contradiction.
\end{proof}

\begin{proof}[Proof of Theorem~\ref{thm:main}]
We fix a positive integer $t$ and we consider a solution $\xg=(x_0,\ldots,x_{d-1})\in W_t(\Z)$ of~\eqref{eq:norm-spe}. Then
\begin{equation}
\label{eq:alpha}
\alpha:=x_0+x_1\xi_t+\cdots+x_{d-1}\xi_t^{d-1}
\end{equation}
has norm $q$. We assume 
\begin{equation}
\label{eq:norm-x}
\Vert \xg\Vert_\infty>t^{C}\vert q\vert^{c}.
\end{equation}
for some $c,C\geq 1$, depending on $\xi$ and $W$, which will be fixed later. We want to bound $t$. We can assume that $\xg$ is primitive (otherwise, write $\xg=\delta\xg'$ with $\xg'$ primitive, and replace $\xg$ by $\xg'$ and $q$ by $q/\delta^d$). Thus
$$
h(\xg)=\log\Vert\xg\Vert_\infty.
$$
We let $k=\Q(\xi_t)$ and we adopt the convention~\ref{convention}. We assume
\begin{equation}
\label{eq:c4-c6}
t\geq\max(c_4,c_6)
\end{equation}
where $c_4$ and $c_6$ are in the statements of Lemma~\ref{lem:AMZ2-Lem3.3} and Proposition~\ref{prop:Qmu+} respectively. 

We partially follow the proof of~\cite[Theorem 2.2]{Am-Ma-Za2} at pages 905--906. Let us first recall the relevant notations. Let $\sigma_1,\ldots,\sigma_d$ be the $\Q(T)$-immersions of $\Q(T,\xi)$ in an algebraic closure $\Q(T,\xi)$. By Lemma~\ref{lem:AMZ2-Lem3.3}(2) we can identify $\sigma_1,\ldots,\sigma_d$ with the $\Q$-immersions of $\Q(\xi_t)$ in $\Qb$ by letting $\sigma_i(\sum a_j\xi_t^j)=\sum a_j(\sigma_i\xi)_t^j$ for $a_0,\ldots,a_{d-1}\in\Q$.

Conjugating~\eqref{eq:alpha} $d$ times we obtain the system $\alphag= A_t\xg$, where $A=(\sigma_i(\xi^j))$. Solving by Cramer's rule we get
$$
x_j=\frac{y_j}\Delta, \hbox{ with } y_j=\Delta_{j,1,t}\sigma_1(\alpha)+\cdots+\Delta_{j,d,t}\sigma_d(\alpha).
$$
Here $\Delta_{j,i,t}$ are $(d-1)\times(d-1)$ determinants with entries among the $\sigma_i(\xi_t^j)$. Since the ultrametric absolute values of the $\xi_t$'s are $\leq 1$, we get
\begin{equation}
\label{eq:alpha2}
\log\Vert\xg\Vert_\infty=h(\xg)=h(\yg)\leq h(\alphag)+c_7\log t
\end{equation}
with $c_7\geq 1$ depending only on $\xi$. We want to apply Lemma~\ref{lem:box}, with $U=(\Z[T,\xi]^*)_t$. To do that, we need a system of multiplicatively independent units of~$U$ of small height. We partially follow the proof of~\cite[Theorem 2.3]{Am-Ma-Za2} at p.910. By Lemma~\ref{lem:AMZ2-Lem3.3}(1) 
$\Z[T,\xi]^*$ has rank $\varrho=\rank \O_{\Q(\xi_t)}^*$. We fix a basis $\gamma^{(1)},\ldots,\gamma^{(\varrho)}$ of $\Z[T,\xi]^*$ modulo torsion. By \cite[Theorem 1' (p.1120)]{Bo-Ma-Za} (in a somewhat modified version: see~\cite[Theorem 3.2]{Am-Ma-Za2}) $\gamma_t^{(1)},\ldots,\gamma_t^{(\varrho)}$ is a basis of $(\Z[T,\xi]^*)_t$. We have 
$$
h(\gammag_t^{(1)})+\cdots+h(\gammag_t^{(\varrho)})\leq c_8\log t
$$
with $c_8\geq 1$ depending only on $\xi$. By Lemma~\ref{lem:box} there exists a unit $\eta\in(\Z[T,\xi]^*)_t$ such that $\alpha':=\eta^{-1}\alpha$ satisfies:
\begin{equation}
\label{eq:alpha'-bis}
h(\alphag')\leq\frac1d\log\vert q\vert+h(\gammag_t^{(1)})+\cdots+h(\gammag_t^{(\varrho)})\leq\frac1d\log\vert q\vert+c_8\log t.
\end{equation}
Let $\gamma\in\Z[T,\xi]^*$ such that $\gamma_t=\eta$. Thus $\gammag$ is in the finite rank subgroup %is in the finite rank group 
$$
\Gamma=\{(\sigma_1(u),\ldots,\sigma_d(u))\;\vert\;u\in\Z[T,\xi]^*\}\subseteq\Gm^d(\F).
$$
with $\F$ be the normal closure of $\Q(T,\xi)/\Q(T)$. By~\eqref{eq:alpha2} and since $h(\alphag)=h(\etag\alphag')\leq h(\etag)+h(\alphag')$, 
\begin{equation}
\label{eq:etag}
%h(\gammag_t)=
h(\etag)\geq h(\alphag)-h(\alphag')\geq \log\Vert\xg\Vert_\infty-c_7\log t - h(\alphag').
\end{equation}

Note that this implies that $\eta$ is not a root of unity. Otherwise, taking into account~\eqref{eq:alpha'-bis}, 
$$
\log\Vert\xg\Vert_\infty\leq \log\vert q\vert+(c_7+c_8)\log t
$$
which contradicts assumption~\eqref{eq:norm-x} provided that 
\begin{equation}
\label{eq:C0}
C\geq c_7+c_8.
\end{equation}

A typical element $v$ in the image $\Gamma'$ of $\Gamma$ by a surjective homomorphism $\Gm^d\rightarrow\Gm$ lies in the group generated by some $\sigma_1(u),\ldots,\sigma_d(u)$. By Assumption~\ref{ass:Thomas-H}(2), $v$ is a root of unity as soon as it is an algebraic number.  This proves that $\Gamma$ is constant-free. A fortiori, it satisfies assumption~\eqref{ass:H2}.

Since $W\subsetneq\A^d$, the vector $\xg$ satisfies an equation $\psi_{t,0}x_0+\cdots+\psi_{t,d-1}x_{d-1}=0$ with $\psig\in\Q(T,\xi)^d$ not zero. 
There exists $c_9\geq 1$, depending only on $\xi$ and $W$, such that $\psig_t\neq0$ when
\begin{equation}
\label{eq:c9}
t\geq c_9
\end{equation}
as we also assume from now on. 
Since $\xg\in\Z^d$, we have $\alphag\in H$, where $H$ is the hyperplane of equation $\theta_{1,t} y_1+\cdots + \theta_{d,t} y_d=0$ with
$$
\thetag=\psig A^{-1}
$$
and where, as before, $A=(\sigma_i(\xi^j))$. Since $\thetag\neq{\bf 0}$, up to renumbering we can assume that there exists $j\geq2$ such that 
\begin{equation}
\label{eq:hyperplane2}
\theta_{1,t}\alpha_1+\cdots + \theta_{j,t}\alpha_{j}=0\hbox{ with no proper vanishing subsums.} 
\end{equation}
We have assumed $\Q(T,\gamma)/\Q$ primitive. We have already remarked that $\eta$ is not a root of unity, and thus in particular $\eta\neq\pm1$. Then $\Q(\eta)=\Q(\xi_t)$ by Proposition~\ref{prop:Qmu+}.  Moreover, by the same proposition, the only possible non trivial subextensions of $\Q(\eta)/\Q$ are quadratic immaginary. On the other hand, $[\Q(\xi_t):\Q]=d$ by Lemma~\ref{lem:AMZ2-Lem3.3}(2) and $d$ is odd by assumption. Thus $\Q(\eta)/\Q$ is primitive as well. 
By Proposition~\ref{cor:height} we then have
$$
h(\etag)\leq 2^{d-2} h(\eta_1,\eta_2)< 2^{d}h(\eta_1,\ldots,\eta_{j}).
$$
Then, by~\eqref{eq:etag},
\begin{multline}
\label{eq:tilde}
h(\alpha'_1,\ldots,\alpha'_{j})\leq h(\alphag')\quad\hbox{and}\\ 
h(\gamma_{1,t},\ldots,\gamma_{j,t})=h(\eta_1,\ldots,\eta_{j})> 2^{-d}\big(\log\Vert\xg\Vert_\infty-c_7\log t - h(\alphag')\big).
\end{multline}
Since $\alphag=\alphag'\etag=\alphag'\gammag_t$, by~\eqref{eq:hyperplane2} we have 
\begin{equation}
\label{eq:hyperplane3}
\alpha'_1\theta_{1,t} \gamma_{1,t}+\cdots + \alpha'_{j}\theta_{j,t} \gamma_{j,t}=0,
\end{equation}
again with no proper vanishing subsums. We want to apply Theorem~\ref{thm:AMZbis}, with~\eqref{eq:hyperplane} replaced~\eqref{eq:hyperplane3}. 
Let $c',C',C''\geq 1$ be the constants which appear in this theorem. We check assumption~\eqref{eq:hgamma_t} of Theorem~\ref{thm:AMZbis}, with $\alphag$ replaced by $\alphag'$ and $P\in\Cu(\Qb)$ the point corresponding to the positive integer~$t$. 
\begin{align*}
h(&\gamma_{1,t},\ldots,\gamma_{j,t})- C'(1+h(t))\\
&\geq h(\gamma_{1,t},\ldots,\gamma_{j,t})-2C'\log t\\
&> 2^{-d}\left(\log\Vert\xg\Vert_\infty - (c_7+2^{d+1}C')\log t- h(\alphag') \right)\qquad\hbox{(by~\eqref{eq:tilde})}\\
&\geq 2^{-d}\Big(\log\Vert\xg\Vert_\infty-\tfrac1d\log\vert q\vert
- \big(c_7+c_8+2^{d+1}C')\big)\log t\Big)\quad\hbox{(by~\eqref{eq:alpha'-bis})}
\end{align*}
and
\begin{align*}
h(&\gamma_{1,t},\ldots,\gamma_{j,t})-c'\,h(\alpha'_1,\ldots,\alpha'_j)\\
%&\geq h(\gamma_{1,t},\ldots,\gamma_{j,t})-c'\,h(\alpha'_1,\ldots,\alpha'_j)-C'\log t\\
&> 2^{-d}\left(\log\Vert\xg\Vert_\infty - c_7\log t-(1+2^d c')h(\alphag')\right)\quad\hbox{(by~\eqref{eq:tilde})}\\
&\geq 2^{-d}\Big(\log\Vert\xg\Vert_\infty-(1+2^d c')\tfrac1d\log\vert q\vert\\
&\hskip 3.2cm - \big((1+2^d c')c_8+c_7)\big)\log t\Big)\quad\hbox{(by~\eqref{eq:alpha'-bis})}.
\end{align*}
We choose
\begin{align*}
C&=(1+2^d c')c_8+c_7+2^{d+1}C',\\
c&=(1+2^d c')\tfrac1d
\end{align*}
(note that $C\geq c_7+c_8$ as required in~\eqref{eq:C0}). Then 
$$
h(\gamma_{1,t},\ldots,\gamma_{j,t})\geq\max\{C'(1+h(t)),c'\,h(\alpha'_1,\ldots,\alpha'_j)\}
$$ 
by assumption~\eqref{eq:norm-x}. Thus~\eqref{eq:hgamma_t} of Theorem~\ref{thm:AMZbis}  is satisfied and this theorem ensures that $\log t=h(t)\leq C''$. Recall that we needed to assume $t\geq\max(c_4,c_6)$ (\cf~\eqref{eq:c4-c6}) and $t\geq c_9$ (\cf~\eqref{eq:c9}). In conclusion Theorem~\ref{thm:main} holds with the above choices of $c$ and $C$ and with 
$$
t_0=\max(\log c_4,\log c_6,\log c_9,C'').
$$
\end{proof}

\section{Justification of the examples.}
\label{sec:ex2}
We now show that Assumption \ref{ass:Thomas-H} holds for (\ref{2}), (\ref{2a}), (\ref{2aa}), (\ref{3}), (\ref{4}), (\ref{4a}) and (\ref{abc}). We do (\ref{4a}) first, as it has more units than the others, and then we sketch the simplifications for (\ref{2a}),(\ref{2aa}),(\ref{4}),(\ref{3}),(\ref{2}) and (\ref{abc}) in turn.

We have
$$X(X^4-1)(X^4-T^4)=X(X+1)(X-1)(X+T)(X-T)(X^2+1)(X^2+T^2)$$
and so
$$u_1=\xi,~~u_2=\xi+1,~~u_3=\xi-1,~~u_4=\xi+T,~~u_5=\xi-T,~~u_6=\xi^2+1$$
are certainly in ${\Z}[T,\xi]^*$ as in part (1). We must now show that they are multiplicatively independent.

We may choose a Laurent expansion
\begin{equation}\label{L}
\xi={1 \over T^4}+{1 \over T^{20}}+{1 \over T^{24}}+{5 \over T^{36}}+{9 \over T^{40}}+{5 \over T^{44}}+{35 \over T^{52}}+{91 \over T^{56}}+{91 \over T^{60}}+{35 \over T^{64}}+\cdots
\end{equation}
at $T=\infty$  (incidentally this allows the specializations $\xi_t$ to be defined simply by replacing $T$ with $t$ so large for convergence). Thus the leading terms of $u_1,u_2,u_3,u_4,u_5,u_6$ are $T^{-4},1,-1,T,-T,1$ respectively.

Therefore a dependence relation $$u_1^{m_1}u_2^{m_2}u_3^{m_3}u_4^{m_4}u_5^{m_5}u_6^{m_6}=1$$ 
clearly implies 
$$-4m_1+m_4+m_5=0.$$
Eliminating $m_5$ gives
\begin{equation}\label{vv}
v_1^{m_1}v_2^{m_2}v_3^{m_3}v_4^{m_4}v_6^{m_6}=\pm 1
\end{equation}
for
$$v_1=u_1u_5^4,~~v_2=u_2,~~v_3=-u_3,~~v_4=-{u_4 \over u_5},~~v_6=u_6.$$
We find that the leading terms of $v_1,v_2,v_3,v_4,v_6$ are all 1, so this too must be on the right-hand side of (\ref{vv}). We can now take formal logarithms to obtain $L=0$ for the linear form in logarithms
$$L=m_1\log v_1+m_2\log v_2+m_3\log v_3+m_4\log v_4+m_6\log v_6.$$
Here we could use Wronskians as in section~\ref{sec:ss} of \cite{Am-Ma-Za2}. But it is faster just to check that the coefficients 
$$m_2-m_3,~~-4m_1+2m_4,~~-{1 \over 2}m_2-{1 \over 2}m_3+m_6,~~-2m_1,~~m_1-{1 \over 4}m_2-{1 \over 4}m_3-{1 \over 2}m_6$$
of $T^{-4},T^{-5},T^{-8},T^{-10},T^{-16}$ in the expansion of $L$ are linearly independent. Therefore $m_1,m_2,m_3,m_4,m_6$ are all zero, and so $m_5$ too, giving the required multiplicative independence. This completes the verification of Assumption \ref{ass:Thomas-H} part (1).

As for part (2), this will follow from Lemma 3.5 of \cite{Am-Ma-Za2} (p.904) provided the Puiseux expansions of all the zeroes of $P(T,X)$ at $T=\infty$ are Laurent series with coefficients in some imaginary quadratic field. In fact this holds with ${\Q}(\sqrt{-1})$; we did not check this sort of thing very rigorously in \cite{Am-Ma-Za2}, but it can be done using Hensel as in Marzenta \cite{mar} (p.293). Thus Assumption \ref{ass:Thomas-H} is done.

We still have to check that there is no field strictly between ${\Q}(T)$ and ${\Q}(T,\xi)$. Now the galois group of $P(T,X)$ over ${\Q}(T)$ is the full symmetric $S_9$. This is just out of the range of Maple 2026; but Habegger suggested applying Dedekind's Theorem to the specialization $P_0=P(0,X)=X^9-X^5-1$. The reduction modulo 2 is irreducible, so $P_0$ is irreducible; and the group of $P_0$ is transitive in $S_9$. The reduction modulo 7 gives an eight-cycle. The reduction modulo 17 gives a transposition $\tau$ multiplied by a disjoint five-cycle, and the fifth power of this is $\tau$. Now any transitive subgroup $G$ of $S_9$ containing a transposition $\tau$ and an eight-cycle, say $\sigma=(12345678)$, must be $S_9$. Namely using transitivity to conjugate $\tau$ we can assume $\tau=(i9)$ for some $i \neq 9$. Then conjugating by powers of $\sigma$ we get $(i9)$ for all $i \neq 9$ inside $G$. Just these transpositions generate $S_9$. The same considerations apply to $P(t,X)$ for any $t$ in $n\Z$ with $n=2\times7\times17=238$. This is not a thin set, so by Hilbert Irreducibility $P(T,X)$ indeed has group $S_9$. One can also argue with ramification.

Now if there was an intermediate field as above it would have to be a cubic extension $K$ of ${\Q}(T)$. Then by looking at galois closures we would find a normal subgroup in $S_9$ of index 3 or 6 according to whether $K$ is galois or not. But it is well-known that the only indices of normal subgroups of $S_n~(n \neq 4)$ are $1,2,n!$. This settles the example (\ref{4a}). 

Incidentally these remarks show that the primitivity condition in Theorem~\ref{thm:main} indeed has a generic character. Namely a general $P(T,X)$ in $\Q[T,X]$ has $S_d$ as group over $\Q(T)$, and an intermediate field of degree $e$ over $\Q(T)$ would lead to a normal subgroup of $S_d$ of index $f$ with $e \leq f \leq e!$. Now $f=1$ implies $e=1$ and $f=d!$ implies $e=d$, both excluded; and then $f=2$ implies $e=2$ contradicting $d$ odd. 

One can also argue without normality. If $G \neq S_d$ is the group over an intermediate field, then the stabilizer $S_{d-1} \neq G$ corresponding to $\xi$ would be contained in $G$. So there is $g$ in $G$ with $g(\xi)=\xi'\neq \xi$. Now $gS_{d-1}g^{-1}$ is the stabilizer $S_{d-1}'$ corresponding to $\xi'$, so $G$ would contain both $S_{d-1}$ and $S_{d-1}'$. As $d \geq 3$ it is easy to see (e.g. via transitivity) that $G=S_d$, a contradiction.

For (\ref{2a}) we can prove the Assumption \ref{ass:Thomas-H} directly using the arguments above. Now the Laurent series have coefficients in $\Q$. Similar direct arguments suffice for (\ref{2aa}) with ${\Q}(\sqrt{-1})$.

Now we do (\ref{4}). We have
$$X^7-T^6X=X(X+T)(X-T)(X^2+TX+T^2)(X^2-TX+T^2)$$
and so we may use
$$u_1=\xi,~~u_2=\xi+T,~~u_3=\xi-T,~~u_4=\xi^2+T\xi+T^2$$
For the independence we choose a Laurent expansion
$$
\xi=-\frac1{T^6}-\frac1{T^{48}}-\frac7{T^{90}}- \frac{70}{T^{132}}
- \frac{819}{T^{174}}-\frac{10472}{T^{216}}- \frac{141778}{T^{258}}-\frac{1997688}{T^{300}}-\cdots,
$$ 
much nicer than (\ref{L}) and incidentally by Lagrange's Theorem it happens to be
\begin{multline*}
-\frac1{T^6}\left(1+\sum_{i=1}^\infty\frac1{i}{7i \choose i-1}\frac1{T^{42i}}\right)\\
=-\frac1{T^6}~_6F_5\left(\frac1{7},\frac27,\frac37,\frac47,\frac57,\frac67,\frac13,\frac12,\frac23,\frac56,\frac76;\frac{7^7/6^6 }{T^{42}}\right)
\end{multline*}
in the standard hypergeometric notation, so that the leading terms of $u_1,u_2,u_3,u_4$ are $-T^{-6},T,-T,T^2$ respectively.

Now a dependence relation $u_1^{m_1}u_2^{m_2}u_3^{m_3}u_4^{m_4}=1$ implies $-6m_1+m_2+m_3+2m_4=0$, and eliminating $m_3$ gives
$v_1^{m_1}v_2^{m_2}v_4^{m_4}=\pm 1$ for
$$v_1=-u_1u_3^6,~~v_2=-{u_2 \over u_3},~~v_4={u_4 \over u_3^2}.$$
As before we find that the leading terms of $v_1,v_2,v_4$ are all 1, and so $L=0$ for the corresponding linear form in logarithms. Here we check that
\begin{equation}\label{Ln}
L={l_1 \over T^7}+{l_2 \over T^{14}}+{l_3 \over T^{21}}+\cdots
\end{equation}
with independent linear forms
$$l_1=6m_1-2m_2-3m_4,~~l_2=-3m_1+{3 \over 2}m_4,~~l_3=2m_1-{2 \over 3}m_2.$$
So we get again Assumption \ref{ass:Thomas-H} part (1).

Part (2) follows again from Lemma 3.5 of \cite{Am-Ma-Za2}, this time with ${\Q}(\sqrt{-3})$. 

For (\ref{3}) we have
$$X^7+27T^6X=X(X^2+3T^2)(X^2+3TX+3T^2)(X^2-3TX+3T^2)$$
and units
$$u_1=\xi,~~u_2=\xi^2+3T^2,~~u_3=\xi^2+3T\xi+3T^2.$$
Now 
$$\xi={1/27 \over T^6}-{1/282429536481 \over T^{48}}+\cdots$$
so we would get $-3m_1+m_2+m_3=0$ in a dependence relation. And now after eliminating $m_3$ just $l_1,l_2$ in (\ref{Ln}) suffice. Also we get ${\Q}(\sqrt{-3})$ as above.

Similarly for (\ref{2}) using
$$X^5-T^4X=X(X+T)(X-T)(X^2+T^2)$$
and 
$$u_1=\xi,~~u_2=\xi+T,~~u_3=\xi-T$$
where $\xi=-1/T^4+\cdots$, this time with ${\Q}(\sqrt{-1})$.

We should not forget the cubic examples (\ref{XT}). The second and third of these are now very easy to treat. For the first the verification of Assumption \ref{ass:Thomas-H} part 2 is also easy directly, even though not all of the conjugates of $\xi$ are Laurent series (of course thanks to Theorem \ref{thm:main-cubic} this is no longer needed; it is enough to check the first sentence in this statement. Here we may note that the corresponding Lagrange-type expansion 
\begin{multline*}
\xi={1 \over T}-{1 \over T^4}+{3 \over T^7}-{12 \over T^{12}}+{55 \over T^{13}}-{273 \over T^{16}}+{1428 \over T^{19}}-{7752 \over T^{22}}+\cdots\\
={1 \over T}~_2F_1\left({1 \over 3},{2 \over 3},{3 \over 2};-{3^3/2^2 \over T^3}\right)
\end{multline*}
is similar to those on p.284 of Chudnovsky \cite{gc2}).

Finally we deal with (\ref{abc}). We failed to prove Assumption \ref{ass:Thomas-H} directly with functional methods (there are too many parameters in $A,B,C$), but it can be deduced from some arguments in \cite{Am-Ma-Za2} as follows.

There we ordered (the distinct) $A,B,C$ so that $A(t)<B(t)<C(t)$ for all sufficiently large $t>0$. 

If $C=B+2$ then we reduce (\ref{abc}) as in \cite{Am-Ma-Za2} (p.911) to the equation
$$x(x-2y)(x+ty)+y^3 = q.$$
This can easily be handled functionally as in the above examples.

If $C=B+1$ then we reduce it as in \cite{Am-Ma-Za2} (p.911) to
$$x(x-y)(x+ty)+y^3 = q$$
which can be similarly handled.

Otherwise we reduce to
$$x(x-r_0y)(x-s_0y)+y^3 = q$$
for $r_0=R(t),~s_0=S(t)$ and $R=B-A,~S=C-A$. Now $1 \leq r_0 \leq s_0-3$ and it was observed in \cite{Am-Ma-Za2} (p.911) that $z(z-r_0)(z-s_0)+1$ is irreducible, and that if $z_0$ is any zero then $z_0,z_0-r_0$ are multiplicatively independent units (again for $t$ large). It follows that $\xi,\xi-R$ are independent, else we could specialise a dependence relation by sending $T$ to $t$ and $\xi$ to $z_0$. This settles Assumption \ref{ass:Thomas-H} part 1, at least for $X(X-R)(X-S)+1$; and it follows at once for $(X-A)(X-B)(X-C)+1$. And part 2 follows in the usual way, because the Laurent series now have coefficients in $\Q$, by Lemma 6.1 (p.910) of \cite{Am-Ma-Za2}.

\section{small solutions.}
\label{sec:ss}
When $t>e^{C'}$, the conclusion $||{\bf x}||_\infty \leq t^{C}|q|^c$ of Theorem \ref{thm:main} seems quite strong. But already we see that for $q=1$ it could be weaker than the conclusion of Theorem \ref{thm:AMZ2-Thm2.2}, which imposes considerable extra functional structure on the solutions (possibly when there are no intermediate fields then there are at most finitely many functional solutions). Nevertheless when we know in advance that $||{\bf x}||_\infty \leq t^{m}$ for some fixed $m$, we can recover a slightly different sort of structure. We will illustrate this just for the example
\begin{equation}\label{qt}
x^3-(t^3-1)y^3=q~~~~~(t \geq 2,q \neq 0)
\end{equation}
as follows.

Using Pad\'e approximants we can construct equations
\begin{equation}\label{pa}
x_n^3-(t^3-1)y_n^3=q_n~~~~~(n=0,1,2,\ldots)
\end{equation}
where $x_n,y_n,q_n$ are polynomials in $t$ depending only on $n$. So these in some sense come from functional structure. Namely the hypergeometric polynomials
\begin{multline}
\label{hyp}
A_n(U)=F(-1/3-n,-n,-2n;U),\\~~B_n(U)=F(1/3-n,-n,-2n;U),
\end{multline}
now with $F=~_2F_1$, lie in ${\Q}[U]$ and have degree $n$. We define $d_n$ as the smallest positive integer such that $d_nA_n,d_nB_n$ lie in ${\Z}[U]$. Then
$$x_n=d_nt^{3n+1}A_n(t^{-3}),~~y_n=d_nt^{3n}B_n(t^{-3})$$
are in $\Z$ and then so is $q_n$ defined by (\ref{pa}).

For example
$$x_0=t,~y_0=1,~q_0=1,$$
$$x_1=3t^4-2t,~y_1=3t^3-1,~q_1=2t^3-1,$$
$$x_2=54t^7-63t^4+14t,~y_2=54t^6-45t^3+5,~q_2=756t^6-756t^3+125.$$

We first dispose of the case
\begin{equation}\label{g1}
|x-sy| \geq 1,~~~~s=\root 3 \of {t^3-1}
\end{equation}
in (\ref{qt}). This will lead to
\begin{equation}\label{mq}
M=\max\{|x|,|y|\} \leq 3|q|.
\end{equation}

If $y=0$ then $M=|x|=|q|^{1/3}$ so nothing more to do. Otherwise $q=(x-sy)(x-\omega sy)(x- \omega^2 sy)$ with $\omega=e^{2\pi i/3}$, and 
$$|x-\omega sy| \geq |\Im(\omega sy)|=\sqrt{3}s|y|/2$$
with the same for $x-\omega^2sy$, so
\begin{equation}\label{qq}|x-sy| \leq M_0
\end{equation}
for $M_0=4|q|/(3s^2|y|^2)$. Thus $|x| \leq s|y|+M_0$ so that $M \leq s|y|+M_0$ too. Also (\ref{qq}) with (\ref{g1}) gives
$$|q|^{1/2} > {3 \over 4}s|y| \geq {3 \over 4}(M-M_0).$$
If $M < 2M_0$ we get
$$M < {8 \over 3s^2}|q| \leq {3 \over t^2}|q|$$
better than (\ref{mq}). And if $M \geq 2M_0$ we get  $|q|^{1/2} \geq {3 \over 8}M$ so $M \leq 3|q|^{1/2}$ even better.

So henceforth we shall assume that $|x-sy| < 1$. This implies $y \neq 0$ and
\begin{equation}\label{mm}
{1 \over 2}s|y|  \leq s|y|-1 \leq |x|\leq M \leq \max\{1+s|y|,|y|\} \leq 2s|y|.
\end{equation}
\begin{proposition} Assume $|x-sy| < 1$. Then for any positive integer $m$ there are $c(m)\geq 1$ and $t_0(m)$ depending only on $m$ with the following properties.

(a) If $M \leq t^m$ and $t \geq t_0(m)$, then either
\begin{equation}\label{p1}M \leq c(m)t|q|
\end{equation}
or there exists $n \leq m/3$ and integers $a \geq 1,b$ with
\begin{equation}\label{p2}ax=bx_n,~ay=by_n,~~a \leq c(m),~|b| \leq c(m)t^{-n}|q|^{1/3}.
\end{equation}

(b) If $t^{m-1} <M \leq t^m$ and $c(m)|q| \leq t^{m-2}$ and $t \geq t_0(m)$, then (\ref{p2}) holds with $n=[m/3]$.
\end{proposition}

\bigskip
Note that (\ref{p1}) is much stronger than $M \leq t^{C}|q|^c$ (an intermediate result $M \leq t^B|q|^{1+\varepsilon}$ can be found in \cite{Za} (p.352).

Note too that in (b) we get $|b| \leq c(m)$ in (\ref{p2}).

We start on the proof. As in (7) of Baker \cite{bah} (p.377) we have
$$A_n(U)-(1-U)^{1/3}B_n(U) = \sum_{k=2n+1}^\infty c_kU^k$$
convergent for $|U|<1$. Putting $U=t^{-3}$ and multiplying by $d_nt^{3n+1}$ we find
$$|x_n| \ll t^{3n+1},~~|y_n| \ll t^{3n},~~|x_n-sy_n| \ll t^{-3n-2}$$
where from now on the implied constants depend only on $n$.

As around (\ref{qq}) we have 
$$|x-sy| \ll {|q| \over t^2|y|^2}.$$ 
Eliminating $s$ gives
$$|x_ny-xy_n| \ll {|y| \over t^{3n+2}}+{|y_n||q| \over t^2|y|^2}$$
which by (\ref{mm}) gives in turn
\begin{equation}\label{za}
|x_ny-xy_n| \ll {M \over t^{3n+3}}+{t^{3n-2}|q| \over |y|^2}.
\end{equation}

To prove part (a) we may assume that $m$ is minimal, so that $t^{m-1} <M \leq t^m$ (in fact as in part (b) above). We choose $n = [m/3]$, so that
$$3n+3 \geq m+1,~~~3n \leq m,$$
and now the first term in (\ref{za}), even taking into account the implied constant, is at most $1/2$ provided $t \geq t_0(m)$.

Thus if $x_ny \neq xy_n$ we get $1 \ll {t^{3n-2}|q| / |y|^2}$ and so
$$M^2 \ll t^{3n}|q| \ll t^m|q| \ll tM|q|$$
leading to (\ref{p1}).

And if $x_ny = xy_n$ we get the first two equations in (\ref{p2}) for $a \geq 0,b$ coprime. But if $a=0$ then $x_n=y_n=0$. Thus $A_n(U),B_n(U)$ would have a common factor $U-t^{-3}$, well-known to be impossible (see for example (8) of \cite{bah} p.377). Therefore $a \geq 1$. In fact $B_n(0)=1$, so actually $A_n(U),B_n(U)$ are coprime. Thus their resultant $r_n \neq 0$ (conjecturally
$$r_n=(-1)^{n(n+1)/2}(3n)^{-n}\prod_{m=1}^{n-1}\left({m^2-{1\over 9} \over n^2-m^2}\right)^{n-m}$$
by the way). Here $r_n=CA_n+DB_n$ for $C,D$ in ${\Q}[U]$, and there is a smallest positive integer $e_n$  such that $e_nC,e_nD$ are in ${\Z}[U]$. Then $\tilde r_n=d_ne_nr_n$ is in $\Z$. As $a$ divides $x_n=d_nt^{3n+1}A_n(t^{-3})$ and $y_n=d_nt^{3n}B_n(t^{-3})$ a quick calculation shows that $a$ divides $\tilde r_nt^N$ for some integer $N=N(n) \geq 0$. But the constant term $f_n$ of $y_n$ (considered as a polynomial in $t$) is non-zero because $B_n$ has degree $n$. Now the highest common factor $u$ of $a$ and $t$ must divide $f_n$. Writing $a=a_1u,t=t_1u$ we find that $a_1$ divides $\tilde r_nu^N$. Thus $a=a_1u \leq \tilde r_nf_n^{N+1}$ as required in (\ref{p2}).

As for $b$ in (\ref{p2}), we have $a^3q=b^3q_n$, so certainly
\begin{equation}\label{bb}|b| \ll |q/q_n|^{1/3}.
\end{equation}
Here $q_n=d_n^3t^{9n+3}\delta_0\delta_1\delta_2$ with
\begin{align*}
\delta_0&=A_n(t^{-3})-(s/t)B_n(t^{-3}),\\
\delta_1&=A_n(t^{-3})-\omega(s/t)B_n(t^{-3}),\\
\delta_2&=A_n(t^{-3})-\omega^2(s/t)B_n(t^{-3}).
\end{align*}
Now
$$|\delta_0|=\left |\sum_{k=2n+1}^\infty c_kt^{-3k}\right| \gg t^{-6n-3}$$
because $c_{2n+1} \neq 0$ (see (7),(10),(11) of \cite{bah}  p.377). Also
$$|\delta_1| \geq |\Im \delta_1| = {\sqrt{3} \over 2}{s \over t}B_n(t^{-3})$$
and as remarked $B_n(0)=1$, so $|\delta_1| \gg 1$ and similarly for $\delta_2$.

Together we get $|q_n| \gg t^{3n}$ and thus by (\ref{bb}) we get the required upper bound for $b$ in part (a) of the Proposition.

In connexion with the above lower bound for $|q_n|$ it might be interesting to know if the polynomial
$$A_n(U)^3-(1-U)B_n(U)^3,$$
already divisible by $U^{2n+1}$, also has a hypergeometric nature as in (\ref{hyp}). But this seems unlikely; for example the above for $n=5$ is
$${2^213 \over 3^{16}}U^{11}\left(1-{5 \over 2}U+{2339 \over 1053}U^2-{292 \over 351}U^3+{10102 \over 85293}U^4-{1331 \over 341172}U^5\right)$$
and 2339 is prime. Also the expression in brackets does not have the form $F(\alpha,-5,\gamma;U)$ for any complex $\alpha,\gamma$, but it is actually the first six terms in the power series expansion of
$$F\left(-{1\over 3}-5,-5,-10;U\right)^2F\left(-{1 \over 3}+6,6,12;U\right).$$

Part (b) follows very quickly. For $m$ is now automatically minimal and so the arguments above give $n=[m/3]$ and some $c(m)$ satisfying (\ref{p2}). If we take $c(m)|q| \leq t^{m-2}$ for this $c(m)$, then (\ref{p1}) is impossible and we get (\ref{p2}). This completes the proof of the Proposition.

We note that equations like (\ref{pa}) can be constructed also with the help of the functional continued fraction for
\begin{multline*}
\root 3 \of {T^3-1}\\=[T;-3T^2,T,-{9 \over 2}T^2,{4 \over 5}T,-{75 \over 14}T^2,{7 \over 10}T,-6T^2,{7 \over 11}T,-{594 \over 91}T^2,{13 \over 22}T,\dots]
\end{multline*}
in the standard notation (note the apparent periodicity of the degrees of the partial quotients, in the spirit of a result of the third author \cite{zcf}). In fact the zeroth approximation $[T;]=T/1$ at $T=t$ gives $x_0/y_0$ as above, while the next $[T;-3T^2]=(3T^3-1)/(3T^2)$ gives something new. But the next $[T;-3T^2,T]$ gives $x_1/y_1$ above. Actually it seems that the approximations of even order $2n$ give the $x_n/y_n$, while the other approximations of odd order $2n+1$ give the ``intermediate''
$${(3n+2)x_n+(12n+6)x_{n+1} \over (3n+2)y_n+(12n+6)y_{n+1}}.$$
Probably one can weaken the hypotheses of the Proposition so that these also turn up. We chose to use Pad\'e because for three variables $x,y,z$ (and more) there may be no convenient analogue of continued fractions.

Thus in the Lombardo-Marzenta example with $\xi_t$ as the real zero of $X^5+4t^4X-1$ and the norm equation
$$
\Norm(x+\xi_ty+\xi_t^2z)=q
$$
one would presumably have to construct a pair
$$
\Norm(x_n+\xi_ty_n+\xi_t^2z_n)=q_n,~~\Norm(x_n'+\xi_ty_n'+\xi_t^2z_n')=q_n'
$$
instead of just (\ref{pa}), presumably via the smallness of the two corresponding linear forms. Then one would have to know that these linear forms cannot be too small (analogous to $|q_n| \gg t^{3n}$ above). This may require some sort of functional Subspace Theorem (known to be effective).

\section{approximation measures.}
\label{sec:irr}

It is classical that bounds for the solutions of~\eqref{eq:ge} lead to irrationality measures for $s=\root 3 \of {t^3-1}$ which improve the Liouville estimate. Theorem \ref{thm:main} (as in~\eqref{eq:gea} above) leads easily to a result in the form
\begin{equation}
\label{eq:gs}
|x-sy| \geq {c \over t^CM^\kappa}
\end{equation}
for $x,y$ in $\Z$ with $M=\max\{|x|,|y|\} \geq 1$; here for some $\kappa<2$ and $c>0,C$ absolute effective. By using the Beukers method and being more careful with heights the third author in \cite{Za} p.352 could reach any $\kappa>1$. But already this precision was observed by Baker in 1964 using hypergeometric functions (this time as introduced by Thue 1908). For example in the Theorem (p.375) of \cite{bah} we may take $m=1,n=3$ and $a=3t^3,b=3(t^3-1)$ to get a result for $\alpha=t/s$. Then~\eqref{eq:gs} follows quickly with $C=3$ and any $\kappa>1$, provided $t$ is large enough with respect to $\kappa$. And since then several people have taken this hypergeometric method further, notably Bennett~\cite{Bennett}.

As noted one can take the exponent of $|q|$ in (\ref{cubest}) as anything greater than 1 and so deduce (\ref{eq:gs}) for the corresponding $s$ with any $\kappa>1$. Here $s$ is the less simple expression
\begin{equation}\label{cardano}
{\root 3 \of {108+12\sqrt{12t^3+81}} \over 6}-{2t \over \root 3 \of {108+12\sqrt{12t^3+81}}}.
\end{equation}

All our new results on diophantine equations can similarly be converted into analogous inequalities involving three or more variables. We give just one example in four variables for the real zero $\xi_t$ of (\ref{1}) with $T=t$. By trivially estimating the non-real conjugates from above, we find that for $t>t_0$
\begin{equation}\label{bf}
|x+\xi_ty+\xi_t^2z+\xi_t^3w| \geq {c \over t^{C}M^\kappa}
\end{equation}
for $x,y,z,w$ in $\Z$ with $M=\max\{|x|,|y|,|z|,|w|\} \geq 1$, now for some $\kappa<4$, with $c>0$ and $C$ absolute effective.

The Liouville estimate gives the same inequality with $M^4$ instead of $M^\kappa$ (and all $t \geq 1$). Just as Baker and Feldman (see for example \cite{baf} p.54) improved by a small amount the Liouville exponent for $x-\alpha y$ with general algebraic $\alpha$ (of degree at least three), this (\ref{bf}) can be regarded as a special analogue in four variables, with an algebraic number $\xi_t$ varying through a parameter $t$, and an explicit dependence on $t$ which is not too bad.

On the other hand the Schmidt Subspace Theorem (see for example \cite{sch} p.181) implies (\ref{bf}) for any $\kappa>3$ (this would be practically best possible), but now ineffectively.

Finally we should mention the papers \cite{gc1},\cite{gc2} of Chudnovsky and \cite{zud} of Zudilin. They focus on the ``irrationality type'', informally defined as the smallest $\kappa+1$ in (\ref{eq:gs}), however without any consideration of possible expressions multiplying the $M^{-\kappa}$. 

The Main Theorem (p.371) of \cite{gc1} obtains $\kappa<2$ for a large class of cubic irrationals (see also pp 380,381 for nice specific examples related to modular functions).

And in \cite{gc2} any $\kappa>1$ is obtained for a large class of parametric cubic irrationalities including (\ref{cardano}); see also \cite{Am-Za} (pp. 1783-1788).

Also in \cite{zud} there is a generalisation to $G$-functions; see also  \cite{Am-Za} (p. 1783).

\section{appendix.}
\label{sec:app}

Here we list the adjustments to \cite{Am-Ma-Za2} necessitated by the changeover from \cite[Assumption 2.1]{Am-Ma-Za2} to Assumption 1.1.

We point out at once that these adjustments come from our incorrect definition of the numbers $r_1$ and $r_2$.\\

The statement of Lemma 3.1 has to be adjusted to involve $\varrho_1$; but the proof is correct and shows this quantity indeed to be constant for large $t$. And of course then so is $2\varrho_2=d-\varrho_1$. The sentence following the proof has to be omitted.

In Lemma 3.3 and its proof $r,r_1,r_2$ have to be replaced by $\varrho,\varrho_1,\varrho_2$ throughout. But the proof then shows that the latter are in fact the former. Part (2) then says that $[\Q(\xi_t):\Q]=d$ for all sufficiently large positive integers $t$, as mentioned just after (\ref{NormForm}). This is a rather strong form of the Hilbert Irreducibility Theorem.

Now that we know that $\varrho=r$ and so on, no adjustments are needed in sections 4 and 5, in particular for the proofs of Proposition 5.4 and Theorems 2.2,~2.3,~2.4,~2.5.

And for the examples in sections 6 and 7, the Assumption 1.1 can be checked directly from the Puiseux series.

\noindent

{\bf F. Amoroso:} Università di Torino, Dipartimento di Matematica, Via Carlo Alberto 10, 10123 Torino, Italy.
 ({\it francesco.amoroso@unito.it}).
\noindent

{\bf D. Masser:} Departement Mathematik und Informatik, Universit\"at Basel, Spiegelgasse 1, 4051 Basel, Switzerland ({\it David.Masser@unibas.ch}).
\noindent 

{\bf U. Zannier:} Scuola Normale Superiore, Piazza dei Cavalieri 7, 56126 Pisa, Italy ({\it u.zannier@sns.it}).

\begin{thebibliography}{[Bo--Gi--SO]}

%\bibitem{Al-Ro} K. Alladi and M. L. Robinson, ``Legendre polynomials and irrationality." {\it J. Reine Angew. Math.} {\bf 318} (1980), 137--155.

\bibitem{Am-Ma-Za} F. Amoroso, D. Masser and U. Zannier, ``Bounded height in pencils of finitely generated subgroups.", {\it Duke Math J.} {\bf 166}, no. 13 (2017), 2599--2642.

\bibitem{Am-Ma-Za2} F. Amoroso, D. Masser and U. Zannier, ``Pencils of norm form equations and a conjecture of Thomas",  {\it Mathematika} {\bf 67}  (2021), 897--916.

\bibitem{Am-Za} F. Amoroso and U. Zannier, ``Irrationality measures for cubic irrationals whose conjugates lie on a curve", {\it Math. Zeit.} {\bf 299} (2021), 1767--1788.

\bibitem{baj} P. Bajpai, ``Effective methods for norm-form equations", {\it Math. Ann.} {\bf 387} (2023), 1271--1288.

\bibitem{bab} P. Bajpai and M.A. Bennett, ``Effective $S$-unit equations beyond three terms: Newman's
conjecture", {\it Acta Arith.} {\bf 214} (2024), 421--458.

\bibitem{bah} A. Baker, ``Rational approximations to $\root 3 \of 2$ and other algebraic numbers", {\it Quart. J. Math. Oxford} {\bf 15} (1964), 375--383.

\bibitem{baq} A. Baker, ``Simultaneous rational approximations to certain algebraic numbers", {\it Proc. Cambridge Phil. Soc.} {\bf 63} (1967), 63--82.

\bibitem{bac} A. Baker, ``Contributions to the theory of Diophantine equations. I. On the representation of integers by binary forms", {\it Phil. Trans. Royal Soc. London Ser. A {\bf 263}} (1967/68), 173--191

\bibitem{baf} A. Baker and G. W\"ustholz, ``Logarithmic forms and diophantine geometry". New Math. Mono. {\bf 9}, Cambridge 2007.

\bibitem{Bennett} M.A. Bennett, ``Explicit lower bounds for rational approximation to algebraic numbers", {\it Proc. London Math. Soc.} {\bf 75} (1997), 63--78.

\bibitem{Bennett2} M.A. Bennett, ``Rational approximation to algebraic numbers of small height : the Diophantine equation $|ax^n-by^n|=1$, {\it J. reine angew. Math.} {\bf 535}  (2001), 1--49.

\bibitem{Be} F. Beukers, ``On a sequence of polynomials", {\it J. Pure and Applied Algebra} {\bf 17/18} (1997) 97--103.

\bibitem{Be-Sc} F. Beukers and H-P. Schlickewei, ``The equation $x + y = 1$ in finitely generated groups". {\it Acta Arith.} {\bf 78} (1996), no. 2, 189--199.

\bibitem{Bo-Ma-Za} E.~Bombieri, D.~Masser, and U.~Zannier, ``Intersecting a curve with algebraic subgroups of multiplicative groups", Int. Math. Res. Notices {\bf 20} (1999), 1119--1140.

\bibitem{sch} E. Bombieri and W. Gubler, ``Heights in diophantine geometry", New Math. Mono. {\bf 4}, Cambridge 2006.

\bibitem{bug} Y. Bugeaud, ``Bornes effectives pour les solutions des \'equations en $S$-unit\'es et des \'equations de Thue-Mahler", {\it J. Number Theory} {\bf 71} (1998), 227--244.

\bibitem{gc1} G.V. Chudnovsky, ``On the method of Thue-Siegel'', {\it Annals Math.} {\bf 117} (1983), 325--382.

\bibitem{gc2} G.V. Chudnovsky, ``The Thue-Siegel-Roth theorem for values of algebraic functions'', {\it Proc. Japan Acad.} {\bf 59A} (1983), 281--284.

\bibitem{Di-Mo} J. Dixon and B. Mortimer, ``Permutation groups". GTM {\bf 163}, Springer-Verlag, New York, 1996.

\bibitem{gal} K. Gy\"ory and L. Lov\'asz, ``Representation of integers by norm forms II", {\it Publ. Math.  Debrecen} {\bf 17} (1970), 173--181.

\bibitem{gyo} K. Gy\"ory, ``Sur une classe des corps de nombres algébriques et ses applications", {\it Publ. Math. Debrecen} {\bf 22} (1975), 151--175.

\bibitem{heu} C. Heuberger, ``On a family of quintic thue equations", {\it J. Symbolic Computation} {\bf 26} (1998), 173--185.

\bibitem{lee} E. Lee, ``Studies on Diophantine equations", Ph.D. thesis, Cambridge University 1992.

\bibitem{lom} D. Lombardo, ``A family of quintic Thue equations via Skolem's $p$-adic method", {\it Riv. Math.
Univ. Parma (N.S.)} {\bf 13} (2022), 161--173.

\bibitem{mar} G. Marzenta, ``Effective resolution of families of norm form equations", {\it Rend. Lincei Mat. Appl.} {\bf 36} (2025), 285--313.

\bibitem{mau} E. Maus, ``Zur Arithmetik einiger Serien nichtaufl\"osbarer Gleichungen 5. Grades", Abh. Math. Sem. Univ. Hamburg {\bf 54} (1984), 227--250.

\bibitem{mat} M. Mignotte and N. Tzanakis, ``On a family of cubics", {\it J. Number Th.} {\bf 39} (1991), 41--49.

\bibitem{sch}  W.M. Schmidt, ``Approximation to algebraic numbers", {\it L'Enseignement Math.} {\bf 17} (1971), 187--253.

\bibitem{tho} E. Thomas, ``Solutions to certain families of Thue equations", {\it J. Number Th.} {\bf 43} (1993), 319--369.

\bibitem{wak} I. Wakabayashi, ``On a family of cubic Thue equations with 5 solutions", {\it Acta Arith.} {\bf 109} (2003), 285--298.

\bibitem{zcf} U. Zannier, ``Hyperelliptic continued fractions and generalized Jacobians", {\it Amer. J. Math.} {\bf 141} (2020), 1--40.

\bibitem{Za} U. Zannier, ``Lecture notes on Diophantine analysis'', EMS Series of Lectures in Mathematics, with an appendix by F. Amoroso, EMS, Z\"urich, 2024.

\bibitem{zud} V.V. Zudilin, ``On a measure of irrationality for values of $G$-functions'', {\it Izvestiya: Mathematics} {\bf 60} (1996), 91--118.

\end{thebibliography}
\end{document}